\documentclass[11pt]{amsart}
\usepackage{amsmath,amsfonts,amssymb,amscd,amsthm,amsbsy,epsf,graphicx, xcolor}
\usepackage{comment}
\usepackage{amsrefs}

\numberwithin{equation}{section}

\newtheorem{thm}{Theorem} 
\newtheorem{lemma}{Lemma}[section]
\newtheorem{lem}{Lemma}
\newtheorem{cor}[thm]{Corollary}
\newtheorem{prop}[thm]{Proposition} 
\theoremstyle{definition}

\def\bye{\end{document}} \def\by{\end{proof}\end{document}}

\def\R{{\mathbb R}}
\def\N{{\mathbb N}}

\DeclareMathOperator{\tr}{tr}

\newcommand{\lbar}[1]{\mkern 1.9mu\overline{\mkern-1.9mu#1\mkern-0.1mu}
\mkern 0.1mu}

\def\mid{\,:\,}
\def\dist{\mathrm{dist}}

\def\USC{\operatorname{USC}}
\def\LSC{\operatorname{LSC}}

\DeclareMathOperator{\argmin}{arg\, min}

\DeclareMathOperator{\e}{\mathrm{e}}

\def\1{\mathbf{1}}
\def\cH{\mathcal{H}}

\def\Lip{\operatorname{Lip}}

\def\gth{\theta}
\def\beq{\begin{equation}}
\def\eeq{\end{equation}}
\def\bproof{\begin{proof}}
\def\eproof{\end{proof}}
\def\bcases{\begin{cases}}
\def\ecases{\end{cases}}
\def\bprop{\begin{prop}}
\def\eprop{\end{prop}}
\def\blem{\begin{lem}}
\def\elem{\end{lem}}
\def\bthm{\begin{thm}}
\def\ethm{\end{thm}}
\def\lab{\label}
\def\fr{\frac}
\def\sep{\sqrt\ep}
\def\bpr{\bproof}
\def\epr{\eproof}
\def\disp{\displaystyle}
\def\wcr{\\[3pt]}
\def\cV{\mathcal{V}}

\date{\today}
 
\def\bbS{\mathbb{S}}  
\def\ep{\varepsilon}
\def\gO{\varOmega}
\def\gG{\varGamma}
\def\tim{\times}
\def\ga{\alpha}
\def\gS{\varSigma}
\def\gd{\delta}
\def\gl{\lambda}
\def\pl{\partial}
\def\gb{\beta}

\def\pbr{\pl_{\mathrm{p}}}

\def\Pi{\varPi}

\author{Hitoshi Ishii$^{1,*}$ and Panagiotis E. Souganidis${}^{2}$}
\thanks{${}^*$ Corresponding author}
\thanks{${}^{1}$ Faculty of Education and Integrated Arts and Sciences, Waseda University, Nishi-Waseda, Shinjuku, Tokyo 169-8050, Japan. 
Partially supported by grants, the JSPS, KAKENHI \#26220702,
\#23340028 and \#23244015} 
\thanks{${}^{2}$ Department of Mathematics, The University of Chicago, 5734 S. University Avenue, Chicago, IL 60657, USA. 
Partially supported by the National Science Foundation grant DMS-1266383} 
\email{hitoshi.ishii@waseda.jp (Hitoshi Ishii),
souganidis@math.uchicago.edu (Panagiotis E. Souganidis)}
\date{\today}

\def\bbS{\mathbb{S}}   

\title[Metastability for parabolic equations with drift]{Metastability for parabolic equations with drift: \\
part II. The quasilinear case}

\begin{document}


\def\gbr{g_{\mathrm{b}}} 
\def\blue#1{\textcolor{blue}{#1}}

\begin{abstract}This is the second part of our series of papers on metastability 
results for parabolic equations with drift. 
The aim is to present a self contained study, using partial differential equations methods, of the metastability properties of quasi-linear 
parabolic equations with a drift and to obtain results similar to those in 
Freidlin and Koralov 
\cites{FK2010,FK2012a}. 
\end{abstract}

\maketitle

\renewcommand{\abstractname}{Notation} 
 
\begin{abstract} 
We work in $\R^n$ and write $\bbS^n$ for the space of real $n\tim n$ symmetric matrices. For any  
$\gth\in(0,\,1]$, $\bbS^n(\gth)$ denotes the subset of all   
$a\in\bbS^n$ satisfying $\gth I\leq a\leq \gth^{-1}I$, where 
$I$ is the $n\tim n$ identity matrix. If  $a\in\bbS^n$, then $\tr a$ denotes its trace,  
and, for  
 $a,b \in\bbS^n$, $a\leq b$ if and only if $b-a$ is a 
nonnegative matrix. Given  $p \in \R^n$, let $ p\otimes p$ denote the 
symmetric matrix $(p_ip_j)_{1\leq i,j\leq n}$. 
If $U$ is a subset of $\R^k$ for some $k\in \N$, then $C(U;\bbS^n(\gth))$ is the set of  $\bbS^n(\theta)$-valued continuous maps from $U$ into $\bbS^n$. For $a\in \bbS^n$ and $p\in \R^n$, $ap\cdot p:= \Sigma_{i,j=1}^n a_{ij}p_jp_i.$
If $r_1,\,r_2\in\R$, then $r_1\wedge r_2:=\min\{r_1,r_2\}$ and $r_1\vee r_2:=\max\{r_1,r_2\}$ and, for $r\in\R$, $r_+=r\vee 0$ and $r_-=(-r)\vee 0$. 
We use the convention  $\inf\emptyset=\infty$ and $\sup\emptyset=-\infty$. 
The  open ball in $\R^n$ with radius $R>0$ and center at $x\in\R^n$ is  $B_R(x)$,
and $B_R:=B_R(0)$.   
Given $\gO \subset \R^n$ and $\gd>0$, we write 
$\gO_\gd:=\{x\in\gO: \text{dist}(x,\partial \gO)\geq \delta\}$,  
and, for $T>0$, 
$Q_T:=\varOmega \times(0,T)$; 
if $T=\infty$, then 
we write $Q$ instead of $Q_\infty$. The parabolic boundary 
of $Q_T$ is $\pbr Q_T:=(\lbar\gO\tim\{0\}) \cup(\pl \gO\tim(0,\,T)).$
We denote by 
$\Lip(A,\R^k)$ the set of the $\R^k$-valued 
Lipschitz continuous functions defined in 
$A\subset \R^k$; when $k=1$, we often write 
$\Lip(A).$
We write $\USC(A)$ and $\LSC(A)$ for the set of, respectively, upper and lower semicontinuous functions defined on $A$,
and, when $A \subset \R^n\times [0,\infty)$ is open, 
$C^{2,1}(A)$ is the space of functions which are continuously differentiable twice with respect to the space variables and once with respect to the time variable. 
Given a bounded family of functions
$f_\delta:A\to \R$,  
${\limsup}^\star_{\gd\to 0} f_\delta (x):=  
\lim_{r\to 0}\sup\{  f_\delta(x+y): x+y \in A,  |y|+\delta \leq r\}$ 
and  ${{\liminf}_{\star}}_{\gd\to 0} f_\delta (x):= 
 \lim_{r\to 0}\inf\{  f_\delta (x+y): x+y \in A,  |y|+\delta \leq r \}.$ 
If $A$ is a closed subset of $\R^n$ and  $f: A \to \R$, $\argmin(f |A):=\{x\in A: f(x)=\min_{y \in A} f(y)\}.$
We use $C$ to denote constants, which may change from line to line. When we want to display the dependence of a constant $C$ on a parameter $a$, we write $C=C(a)$, and, 
for $a,b \in \R$, $a\approx b$ means that $a$ and $b$ are close to each other in a controlled way. Finally to simplify the notation we write $\{a_k\}$  to denote the sequence $\{a_k\}_{k\in \N}$.
\end{abstract}

\section{Introduction} 

 This is the second part of our series of papers on metastability 
results for parabolic equations with drift. 
The aim is to present a self contained study, using partial differential equations 
(pde for short) methods, of the metastability properties of quasi-linear 
parabolic equations with a drift and to obtain results similar to those in 
Freidlin and Koralov \cites{FK2010,FK2012a}. 
\smallskip

 More precisely we are interested in the asymptotic behavior, as $\ep\to 0$ and $t\to \infty$, of the solution $u^\ep=u^\ep(x,t)$ of the 
initial-boundary value problem 
\beq\lab{eq:ql-pde}
u^\ep_t=\ep\tr[a(x,u^\ep)D^2u^\ep]+b(x)\cdot Du^\ep 
\ \text{ in  }  \ Q,
\eeq
and 
\beq\lab{eq:ibv-g}
u^\ep=g \ \text{ on } \ \pbr Q, 
\eeq
where 
\begin{equation}\lab{omega}
\varOmega \text{ is a bounded $C^1$-domain with outward  normal vector $\nu$}
\end{equation}
and
\begin{equation}\label{g}
g\in C(\lbar \gO).
\end{equation}

 Throughout the paper we assume that, for some  $\gth_0\in(0,\,1]$,  
\beq\lab{ellipticity}
a\in C(\lbar\gO\tim\R;\bbS^n(\theta_0)), 
\eeq
and  
\begin{equation}\label{blip} 
b\in\Lip(\R^n;\R^n)  \ \text{ with } \ b(0)=0
\end{equation}
is such that 
\beq\lab{g-astable}
\text{\begin{minipage}[adjusting]{0.7\textwidth}
the origin is a (unique) globally asymptotically stable point
 of the dynamical system $\dot X=b(X)$ generated by $b$.
\end{minipage}} 
\eeq 

This last assumption is further quantified by the additional requirements that $b$ points inward at the boundary points of $\gO$, that is, 
\beq\lab{b-inward}
b\cdot \nu<0  \ \text{ on } \ \pl\gO, 
\eeq
and there exist $b_0>0$ and $r_0>0$ such that $\lbar B_{r_0}\subset\gO$, and 
\beq\lab{borigin}
b(x)\cdot x\leq -b_0|x|^2 \ \  \text{ for all } \ \  x\in B_{r_0}.
\eeq

 For later use we summarize all the above assumptions 
in the  list
\beq\label{B}
\text{\eqref{omega}, \eqref{g}, \eqref{ellipticity}, \eqref{blip}, \eqref{g-astable}, \eqref{b-inward}  and \eqref{borigin}.}
\eeq

 The asymptotic behavior of the $u^\ep$'s  is described in Theorem \ref{thm:FK}. Our arguments  are based entirely  
on pde methods and  the main tools are the comparison principle 
and the construction of two kinds of barrier functions 
for parabolic equations. The later was   
the main subject of our previous paper 
\cite{IS2015}.    
\smallskip

 We work with either classical or viscosity solutions depending on the context and
 most of the times we say solution without making a distinction.  When we write inequalities for viscosity sub- or super-solutions, we use the $\leq$ and $\geq$ signs for a sub- and super-solutions respectively.  Finally, we will always work with $\ep \in (0,1)$ and we will not repeat this.

\smallskip

 An important tool  is the quasi-potential $V^c$ associated, for each $c\in \R$,  with
$(a(\cdot,c),b)$, which is characterized by the property  
$$V^c  \ \text{is the maximal subsolution of  $H^c(x,Du)=0$ in $\gO$ 
and $u(0)=0$},$$
where the Hamiltonian $H^c\in C(\lbar\gO\tim\R^n)$ 
is given by $$H^c(x,p)
:=a(x,c)p\cdot p+b(x)\cdot p.$$ 
\smallskip

The quasi-potential $V^c$ is also the  unique (viscosity) solution 
$u\in\Lip(\lbar\gO)$ of the state-constraints problem for the Hamilton-Jacobi 
equation $H(x,Du)=0$ in $\gO$, with the  additional condition that $u(0)=0$.  
(See Lemma~\ref{lem:A-comparison} 
in Appendix~C for the uniqueness of this state-constraints problem, and  
also Soner \cite{So1986}, Fleming and Soner \cite{FS2006} and  Ishii \cite{Is1989} for some related aspects.)

 Next we introduce some terminology 
and introduce some additional notation and hypotheses similar to those in 
\cites{FK2010, FK2012a}.

Consider the map  $M : \R\to \R$ given  by
\beq \lab{FK8}
M(c):= \min_{\pl\gO}V^{c}.
\eeq

The continuity of $a$ and the stability properties of viscosity solutions yield that the functions $ \R  \owns c \mapsto M(c)$ and 
$\lbar \varOmega \times  \R \owns (x,c) \mapsto V^c(x)\in \R$ are continuous.
The continuity of the latter is an easy 
consequence of the uniqueness of the state constraints problem.
\smallskip

Given $g \in C(\lbar Q)$, we set
\[
c_0:=g(0), \quad g_{\min}: = \min_{\lbar\gO}g, \quad
g_{\max}: = \max_{\lbar\gO}g, \quad
g_1: = \min_{\pl\gO}g, \quad 
g_2: = \max_{\pl\gO}g, 
\]
and note that $[g_1,\, g_2]\subset [g_{\min}, \,g_{\max}]$. 
Henceforth we write 
$$I_g:=[g_{\min}, \,g_{\max}],$$  
and we introduce the multi-valued map   
 $G\mid I_g\to 2^{\R}$ by
\[
G(c):=\{g(x)\mid x\in\argmin(V^c|\pl\gO)\}. 
\]

It is immediate that $G(c)\subset I_g$ for all $c\in I_g$. Moreover, since $(c,x)\mapsto V^c(x)$ and $g$ are continuous on $\R\tim\pl\gO$ 
and $\pl\gO$ respectively, it is easily checked that 
$G$ is upper semicontinuous on $I_g$ and, hence, 
$G(c)$ is compact for all $c\in I_g$. 

Next we define the functions $G^\pm\mid I_g\to I_g$ by
\[
G^+(c):=\max G(c) \ \ \ \text{ and } \ \  \ G^-(c):=\min G(c).
\] 
and note that they are respectively
upper and lower 
semicontinuous, and, moreover,  
\[
G^+(c)=\max_{\argmin(V^c|\pl\gO)}g \ \ \text{ and } \ \ G^-(c)=\min_{\argmin(V^c|\pl\gO)} g.   
\]

Following  \cites{FK2010, FK2012a}, we assume that 
\beq\label{FK13}
G^+(c_0)=G^-(c_0), 
\eeq
and set 
$$g_0:=G^+(c_0)=G^-(c_0).$$
  
This assumption means that the set $G(c_0)$ is a singleton, that is, 
$$g(x)= g_0 \ \ \text{ for all  } x\in \argmin(V^{c_0}|\pl\gO).$$ 

Next we define $c_1$ as follows:
\beq
\left\{
\text{
\begin{minipage}{0.7\textwidth}
if $g_0\geq  c_0$, then $c_1:= \inf\{c\in [c_0,\infty)\mid G^-(c)\leq c\}$, and,\\[2mm] 
if $g_0\leq c_0$, then  $c_1:= \sup\{c\in (-\infty,\,c_0]\mid
G^+(c)\geq c\}.$
\end{minipage}
}
\right.
\eeq

Note that, since $G(I_g):=\bigcup_{c\in I_g}G(c) \subset[g_1, g_2]$, we always have
$c_1 \in[g_1,\, g_2]$ and observe that 
\beq\label{FK12'}
\left\{
\text{
\begin{minipage}{0.7\textwidth}
if $c_1>c_0$, then $\,G^-(c)>c \,$ for all $\, c\in[c_0,\,c_1)$, \\[3pt]
if $c_1<c_0$, then $\,G^+(c)<c \,$  for all $\,c\in(c_1,\,c_0]$.
\end{minipage}
}
\right.
\eeq

We assume that the graph of
$G$ crosses the diagonal from the left to the right at $c_1$, that is   
\beq\lab{FK12}
\begin{cases}
\text{for all $\gd_0 > 0$, there exists $\gd\in (0,\,\gd_0]$ such that}\\[1mm]
\qquad \text{if $c_0\geq c_1 > g_{\min}$,  then $G^-(c_1-\gd) > c_1-\gd,$} \\[1mm]
\qquad \text{if $c_0\leq c_1 < g_{\max}$,  then $G^+(c_1 + \gd) < c_1 + \gd,$}
\end{cases}
\eeq
and we define the function $\bar c\mid (0,\,\infty)\to I_g$ 
as follows:
For each  $\gl\in(0,\,\infty)$, 
\beq\label{takis}
\bar c(\gl):=\bcases
c_0  \ \text{ if either} \ \gl<M(c_0) \ \text{ or } \ c_1=c_0,\\[.75mm]
\min(c_1,\inf\{c\in[c_0,c_1]:M(c)=\gl\})  \ \text{ if } \ 
\gl\geq M(c_0) \  \text{ and } \ c_1>c_0,
\\[.75mm]
\max(c_1,\sup\{c\in[c_1,\,c_0]:M(c)=\gl\}) \ 
\text{ if } \ \gl\geq M(c_0) \ \text{ and }  \ c_1<c_0.
\ecases
\eeq
\smallskip 

 For later use we summarize the above assumptions 
in the  list
\beq\label{G}
\text{\eqref{FK13} and \eqref{FK12}.}
\eeq

Since the definition of $\bar c(\gl)$ is cumbersome, for clarity and to compare with the linear problem, we discuss
what happens when $a(x,c)$ is independent of $c$. 
In this case the quasi-potential $V$ and, hence,
its minimum value $M=\min_{\pl\gO}V$ do not depend on $c$, and 
the multi-valued map $G$ is a constant. 
Assumption \eqref{FK13} then states that 
$g_0=\min_{\argmin(V|\pl\gO)}g=\max_{\argmin(V|\pl\gO)}g$ and 
$G(c)=\{g_0\}$ and $G^-(c)=G^+(c)=g_0$ for all $c\in I_g$.
It is easily checked that, if $g(0)={g_0}$, 
then $\bar c(\gl)=g(0)={g_0}$ for all $\gl>0,$ and, if either $g(0)<{g_0}$ or $g(0)>{g_0}$, 
\[\bar c(\gl)=
\bcases
c_0&\text{ if }\ \gl\leq M, \\[.5mm]
c_1&\text{ if }\ \gl>M,
\ecases
\]     
while, 
if $g(0)\not={g_0}$, then $\bar c(\gl)$ is discontinuous at $\gl=M$.
\smallskip 

 The main result, which is similar to  
\citelist{\cite{FK2010}*{Theorem 3.1}\cite{FK2012a}}, is:

\begin{thm}\lab{thm:FK}
Assume  \eqref{B} and \eqref{G} 
and let  $\lambda >0$ be a point of continuity of $\bar c$. If, for $\ep\in(0,\,1),$ 
$u^\ep\in C(\lbar Q)\cap C^{2,1}(Q)$ is  a solution 
of \eqref{eq:ql-pde} and \eqref{eq:ibv-g}, then, for each $\delta >0$ so that $\varOmega_\delta \neq \emptyset$,   
\[
\lim_{\ep\to 0}u^\ep(\cdot,\exp(\gl/\ep))=\bar c(\gl)
\  \text{ uniformly in } \  \gO_\gd. 
\]
\end{thm}

In view of the previous discussion,  
when $a(x,c)$ is independent of $c$, that is for linear equations, 
Theorem \ref{thm:FK} is the same as  
\cite{IS2015}*{Theorem 1}, except 
if $g(0)=g_0$. In this case, 
\cite{IS2015}*{Theorem 1} asserts, in addition, the uniform convergence of 
$u^\ep(\cdot,\exp(\gl/\ep))$ on any compact subset of $\gO\cup\argmin(V|\pl\gO)$.     
\smallskip

As in \cites{FK2010, FK2012a}, 
to prove Theorem~\ref{thm:FK} we need to show the following 
three propositions, which were proved in \cite{FK2012a}  using  large deviation results from \cite{FW2012}.
The first two 
together state \cite{FK2012a}*{Lemma~3.11}, while the third  
is an observation which is very crucial for the proof of Lemma 
\ref{lem:main'} (see \cite{FK2012a}*{Lemma~3.12}).

\begin{prop}\lab{thm:A} Assume  \eqref{B} 
and  
let $u^\ep\in C(\lbar Q)\cap C^{2,1}(Q)$ be a solution of \eqref{eq:ql-pde}. Assume furthermore  that  
the $u^\ep$'s are bounded on $Q$  uniformly on $\ep$ and 
suppose that there exist sequences $\{\mu_k\},\, \{\gl_k\} \subset (0,\,\infty)$ and $\{\ep_k\}\subset(0,\,1)$ and constants $0<a_1<a_2$ and $\beta_1, \beta_2\in \R$  such that  
$\lim_{k\to\infty}\ep_k=0$, and, for all $k\in\N$, 
\[0<a_1\leq \mu_k <\gl_k\leq a_2, \ \ 
u^{\ep_k}(0,\exp(\mu_k/\ep_k))=\gb_1 \ \text{ and } \ \  
u^{\ep_k}(0,\exp(\gl_k/\ep_k))=\gb_2.
\] 
If $\gb_1\not=\gb_2$, then 
$\,\limsup_{k\to \infty}\gl_k \geq M(\gb_2).$
\end{prop}

\begin{prop}\lab{thm:B} Assume \eqref{B} and 
 let $u^\ep\in C(\lbar Q)\cap C^{2,1}(Q)$ be a solution of \eqref{eq:ql-pde} and \eqref{eq:ibv-g}.  
Assume further that there exist sequences 
 $\{\mu_k\},\, \{\gl_k\} \subset (0,\,\infty)$ and $\{\ep_k\}\subset(0,\,1)$
and constants $0<a_1<a_2$ and $\beta_1, \beta_2 \in I_g$ such that 
$\lim_{k\to\infty}\ep_k=0$, and, for all $k\in \N$,
\[0<a_1\leq \mu_k <\gl_k\leq a_2, \ \ 
u^{\ep_k}(0,\exp(\mu_k/\ep_k))=\gb_1 \ \ \text{and} \ \  
u^{\ep_k}(0,\exp(\gl_k/\ep_k))=\gb_2.
\]
If $\,\gb_1<\gb_2,$ 
then $\,G^+(\gb_2)\geq \gb_2$,
and, if $\,\gb_2<\gb_1$, 
then 
$\,G^-(\gb_2)\leq \gb_2.$ 
\end{prop}

\begin{prop}\lab{thm:C} Assume 
\eqref{B} and 
let $u^\ep\in C(\lbar Q)\cap C^{2,1}(Q)$ be a solution of \eqref{eq:ql-pde} and \eqref{eq:ibv-g}.  
Fix $\gb_0\in I_g$ and $\rho_0>0$, and assume that,
for any $\gd>0$, there exist $\gamma>0$ and a sequence  
$\{\ep_k\}\subset (0,\,1)$ such that $\lim_{k\to\infty}\ep_k=0$ and, for all 
$\rho\in[\rho_0-\gamma,\,\rho_0+\gamma]$ and $k\in\N$,
\beq\label{eq:C1}
u^{\ep_k}(0,\,\exp(\rho/\ep_k))\in[\gb_0-\gd,\,\gb_0+\gd]. 
\eeq
If either 
\beq\label{eq:C2}
G^-(\gb_0)> \gb_0 \ \ \text{ or } \ \ G^+(\gb_0)<\gb_0,
\eeq
then $\rho_0\leq M(\gb_0)$.  
\end{prop}

We discuss next some of the new ideas that are needed to prove the main theorem.  
\smallskip

Recall that we are interested in the asymptotic behavior,  as $(\ep,t)\to (0,\infty)$, of 
the solution
$u^\ep$ of  \eqref{eq:ql-pde} and 
\eqref{eq:ibv-g} in a logarithmic time scale,  
that is, in the behavior, as $\ep\to 0$, 
of $u^\ep(x,\exp(\gl/\ep))$ for any fixed $\gl>0$.
It turns out that this is a consequence of what we call ``uniform asymptotic constancy'' which yields that, as $t\to \infty$,  
$u^\ep(\cdot,t)$  behaves similarly to  $u^\ep(0,t)$ in the space $C(\gO)$ equipped with the locally uniform convergence topology, 
\smallskip

The uniform asymptotic constancy (see Theorem~\ref{thm:conclusion3} below) is a crucial observation that goes beyond \cite{IS2015}.   
Roughly it  says that,  
if $u^\ep$ is a bounded 
solution of \eqref{eq:l-pde}, then, as $\ep\to 0$,  
for any compact $K\subset \varOmega$ and $\gd>0$,
\[
u^\ep(x,t)\approx u^\ep(0,t) \ \text{ uniformly for }\ (x,t)
\in K\tim [\e^{\gd/\ep},\,\infty). 
\] 

With this fact at hand 
the main theorem (Theorem~\ref{thm:FK}) 
is an easy consequence of Propositions~\ref{thm:A}, \ref{thm:B} and \ref{thm:C}. 
\smallskip

Their proofs  
are based on the comparison (or maximum) 
principle and, thus, on the construction of barriers, that is sub- and  super-solutions 
of \eqref{eq:ql-pde}. We have already built such  functions 
in  our  previous work \cite{IS2015}, where the matrix $a(x,c)$ 
is independent of $c$. Here (see Proposition~\ref{thm:long} and Corollary \ref{cor:long}) we modify 
the construction of one class of barrier 
functions in order to make the comparison argument 
straightforward.   
\smallskip

The building block of  the  barrier functions in \cite{IS2015} 
and here is viscosity solutions 
of $H_\ga(x, Du)=0$ with some additional normalization conditions, 
where $\ga\in C(\lbar\gO;\bbS^n(\gth_0))$ is 
is selected as explained below
and $H_\ga(x,p):=\ga(x)p\cdot p+\b(x)\cdot p$. 
An important observation is that, if $V_\ga$ is  the quasi-potential associated with $(\ga,b)$, 
then $V_\ga>0$ in $\lbar\gO\setminus\{0\}$ and
$M_\ga:=\min_{\pl\gO}V_\ga>0$.
\smallskip

The barriers $w^\ep:\lbar Q \to \R$  are supersolutions of 
\eqref{eq:ql-pde} of  the form 
\[
w^\ep(x,t):=\exp\left(\fr{v(x)-m}{\ep}\right)+d_\ep t,
\]   
where $m$ and $d_\ep$ are positive constants such that  
$0<m<M_\ga$ and 
$d_\ep = \exp(-\gl_\ep/\ep)$ for some 
$\gl_\ep\approx m$,
and  $v$ is an appropriately chosen smooth approximation of 
$V_\ga$.  The choice of $m$ yields that, for $\ep$ sufficiently small, $w^\ep$ is compatible with  the Dirichlet data $g$ on $\pl\gO \tim[0,\,\infty)$.  

In view of the fact that a priori we 
have little knowledge of the uniform in $\ep$ regularity of solutions of 
\eqref{eq:ql-pde}, given such a solution $u^\ep$, we treat 
$a^\ep =a(x,u^\ep(x,t))$ as an arbitrary element 
of $C(\lbar Q;\bbS^n(\gth_0))$.  
\smallskip

To motivate the choice of $\ga$ in the construction of the barrier function  given the  $a^\ep$ above we compute in $Q$
\begin{align*}
w^\ep_t&-\ep\tr[a^\ep(x,t)D^2 w^\ep]
-b\cdot Dw^\ep   
\\&=d_\ep-\ep^{-1}\exp\left(\fr{v(x)-m}{\ep}\right) 
(H_{\ep}(x,t, Dv)+\ep \tr[a^\ep D^2v]) 
\end{align*}
with $H_{\ep}(x,t,p):=a^\ep(x,t)p\cdot p+b(x)\cdot p$. 
\smallskip

If $\ga\in C(\lbar\gO;\bbS^n(\gth_0))$
satisfies $a^\ep\leq \ga$ in  $Q$, 
then 
\[
w^\ep_t-\ep\tr[a^\ep(x,t)D^2 w^\ep]
-b\cdot Dw^\ep 
\geq d_\ep-\ep^{-1}\exp\left(\fr{v(x)-m}{\ep}\right)
(H_{\ga}(x, Dv)+O(\ep))\geq 0,
\]
with the last the inequality holding, if $\ep$ is sufficiently small, 
because of the choice of 
$v$ and $d_\ep$ --the details are given in Proposition~\ref{thm:short}. 
\smallskip

A very important fact in our analysis (see Proposition 
\ref{thm:quasi2-1} below for the precise statement) is that 
the locally uniform convergence topology of $C(\gO)$ is strong enough
to imply that, if $\ga(x)\approx a(x,c)$ in $C(\gO)$, then 
$M_\ga\approx M(c)$ and $\argmin (V_\ga\,|\,\pl\gO)
\approx \argmin (V^c\,|\,\pl\gO)$.  
\smallskip

To describe the idea which is in the core of the proof of, for example,  
Proposition~\ref{thm:A}, 
 we consider the very special case that,  for $\ep>0$ sufficiently small and some constants $c$, $\gamma>0$ and $0<\gd<\mu<\gl$,
\[
|u^\ep(0,t)-c|<\gamma  \ 
\text{ for all \ $t\in [\exp(\gd/\ep),\,\exp(\gl/\ep)],$}
\]
and
\[
u^\ep(0,\exp(\gd/\ep))=c \ \ \text{ and  } \ \ u^\ep(0,\,\exp(\mu/\ep))>c+\eta \ \text{ 
for some $\eta\in (0,\gamma)$.}
\]

We then choose $\ga  \in C(\lbar \varOmega;\bbS^n(\gth_0))$ so that 
$a^\ep\leq \ga$  in  $ \gO\tim[t_\ep,\,T_\ep]$, where 
$t_\ep:=\exp(\gd/\ep)$ and $T_\ep:=\exp(\gl/\ep)$.  
Using the barrier $w^\ep$ as in the linear case 
(see \cite{IS2015}*{Theorem 1 (i)}),  
we conclude 
that, as $\ep\to 0$, for any $\rho<M_\ga$, 
$u^\ep(0,t) \to c$ for all $t\in [t_\ep,\,T_\ep \wedge \exp(\rho/\ep)]$, 
which implies that $\mu\geq M_\ga$. 
Furthermore, 
according to the previous arguments,  $\ga$ can be  chosen, 
so that,  as  $\gamma \to 0$, $M_\ga \to  M^c$.

\subsection*{Organization of the paper} The rest of the paper is organized as follows. In Section 2 we study 
 the asymptotic constancy, that is the effect  
of the drift term in parabolic equations like 
\eqref{eq:ql-pde}. In Section 3 we introduce Hamilton-Jacobi equations 
related to \eqref{eq:ql-pde}, which have quadratic nonlinearity, and study the continuity properties of the associated quasi-potentials.   
Section 4 is devoted to the construction of two kind of barrier functions, 
or sub- and super-solutions, which are used  to study  
 the asymptotic behavior of solutions of  linear parabolic equations, 
that is  equations like  \eqref{eq:ql-pde} with $a \in C(\lbar Q;\bbS^n(\gth_0))$. 
The proofs of Propositions ~\ref{thm:A}, \ref{thm:B} and \ref{thm:C}  
and Theorem~\ref{thm:FK} are given in Sections 5 and 6 respectively. 
Some basic properties of viscosity solutions are explained in Appendices 
A, B and C.

\section{The Asymptotic constancy} 

 We consider the linear pde
\begin{equation}\label{eq:l-pde}
u_t^\ep=\ep\tr[a^\ep(x,t) D^2u^\ep]+b(x)\cdot Du^\ep 
\ \ \ \text{ in }\ Q. 
\end{equation}

 We assume, in addition to  \eqref{blip}  and \eqref{borigin}, that 
\begin{equation}\label{takis1}
 a^\ep\in C(\lbar Q, \bbS^n(\gth_0)).
\end{equation}

 The  goal  here is to show that  the drift term in 
\eqref{eq:l-pde} has a strong effect to propagate, as $\ep\to 0$,  the values of the solutions $u^\ep$ at $x=0$ 
to $\gO$; for future reference  we call this fact the asymptotic constancy.
\smallskip 

 It turns out that the asymptotic constancy does not depend on any properties of $a^\ep$ other than \eqref{takis1}. It is, therefore, technically more convenient to study, in some instances, 
instead of \eqref{eq:l-pde}, the problem 
\beq\lab{eq:Pucci}
v_t=\ep P^+(D^2v)+b(x)\cdot Dv \ \ \ \text{ in }\ Q,
\eeq
where  $P^+$ is  the  Pucci operator associated with $\bbS^n(\gth_0)$ defined by 
\beq\lab{eq:Pucci1}
P^+(X)=\sup\{\tr[AX]: A\in\bbS^n(\gth_0)\},
\eeq
which is, obviously, uniformly elliptic with constants
$\theta_0$ and $\theta_0^{-1},$ that is, for all matrices $X, Y \in \bbS^n$ such that $X\leq Y,$
\beq\label{takis2}
\theta_0 \tr(Y-X) \leq 
P^+(Y)-P^+(X) \leq \theta_0^{-1} \tr(Y-X).   
\eeq

\subsection*{Some useful barrier functions}\lab{sec:barrier} 
We fix an auxiliary function  $h\in C^2([0,\,\infty))$ with  
the properties
\beq\lab{eq:h}
0\leq h \leq 1,  \ h=0 \ \text{ in } \ [0, 1/2], \  h=1 \ \text{ in } \ [1,\infty) \ \text{and} \ h'\geq 0,
\eeq
set 
\[k:=b_0/ 2  \ \  
\text{ and }  \  \ R_0:= 2 \sqrt{2n}/\sqrt{b_0\theta_0},
\]
choose  $R\in [R_0,\infty), r \in (0,r_0]$, where $r_0$ is as in \eqref{borigin},  and  $\ep_0\in(0,\,1)$ 
so that 
\beq\lab{eq:1.4}
\sqrt\ep_0 R<r,
\eeq
and, for  $\ep\in(0,\ep_0]$, let 
\beq\lab{eq:tau}
\tau=\tau(\ep):=\frac{1}{k}\log\left(\fr{r}{R\sqrt\ep}\right).
\eeq 

 With all these choices at hand we introduce the functions 
 $p^\ep, q^\ep:  \R^n\tim[0,\,\infty)\to\R$ defined by 
\beq\lab{eq:p}
p^\ep(x,t):=h((R\sqrt\ep)^{-1}|x|\e^{-kt})
\eeq
and
\beq\lab{eq:q}
q^\ep(x,t):=p^\ep(x,t)+\fr{\|h''\|_{L^\infty}}{R^2\gth_0}\int_0^t
\e^{-2ks}\mathrm{d} s;
\eeq
observe that, since $h$ vanishes 
identically in a neighborhood of  
the positive time axis $l:=\{0\} \times (0,\infty)$, $p^\ep$ and $q^\ep$ are smooth 
in $\R^n\tim(0,\,\infty)$. 
\smallskip

 We note  that  $p^\ep$ appears 
in the proof of \citelist{\cite{FK2010}*{Lemma 3.6}\cite{FK2012a}}.  The  difference is that these references consider equations like \eqref{eq:l-pde}, while here 
we study \eqref{eq:Pucci}. 
\smallskip

 The following lemma summarizes the properties of $q^\ep$. Its proof is based on long explicit but also straightforward calculations. The reader may want to skip the details on first reading.

\begin{lem}\lab{lem:q} Assume \eqref{blip}, \eqref{borigin} and  \eqref{takis2}. With the above choices of $\,k$, $R$, $r$, $\ep_0$, $\ep$ and $\tau$, 
the function $q^\ep $ given by \eqref{eq:q} is a supersolution to \eqref{eq:Pucci} in  
$B_{r_0}\tim (0,\,\infty).$
Moreover,  
\[\begin{cases}
q^\ep(\cdot,0) \geq 0 \ \text{ in  } \  B_r,  \quad 
q^\ep(\cdot,0)\geq 1 \  \text{ in} \  B_r\setminus B_{\sep R},\\[1.2mm]
q^\ep\geq 1 \  \text{ in} \ \pl B_r\tim[0,\,\tau] \ \text{ and } \ 
q^\ep(\cdot ,\tau)\leq \fr{\|h''\|_{L^\infty}}{b_0\gth_0R^2} \ \text{ on } \ B_{r/2}.
\end{cases}
\]
\end{lem}

\begin{proof}
First note that 
\[
p^\ep(x,t)=1 \ \  \ \text{ if } \ |x|\geq R\sqrt\ep \e^{kt}
\quad\text{ and }\quad
p^\ep(x,t)=0 \ \ \ \text{ if } \ |x|\leq \fr 12 R\sqrt\ep \e^{kt}. 
\]

For  $(x,t)\in B_{r_0}\tim(0,\,\infty)$ we write 
 $\rho=\fr{1}{R\sqrt \ep}$, $r_{x,t}=(R\sqrt\ep)^{-1}|x|\e^{-kt}$ and $\bar x:=x/|x|$ (since, in view of the above, $p^\ep$ vanishes in a neighborhood of the origin we do not
 have to be concerned about $x=0$),
and find 
\[
\begin{aligned}
p^\ep_t(x,t)&\,=-kh'(r_{x,t})|x|\rho\e^{-kt},
\qquad 
Dp^\ep(x,t)=h'(r_{x,t})\rho\bar x\e^{-kt},\\
D^2p^\ep(x,t)&\,=h'(r_{x,t})\rho \e^{-kt}\fr{1}{|x|}(I-\bar x\otimes \bar x)
+h''(r_{x,t})\rho^2\e^{-2kt} \bar x\otimes\bar x.
\end{aligned}
\]

Moreover, for any $a \in\bbS^n(\gth_0)$ and all $(x,t)\in \overline Q$ with $x \neq 0$, we have
\[ |\tr [a(I-\bar x\otimes \bar x)]| \leq \gth_0^{-1}(n-1)<\gth_0^{-1}n
\quad\text{ and }\quad
|\tr [a\bar x\otimes\bar x]|\leq \gth_0^{-1},
\]
and, therefore,
\[
\begin{aligned}
p^\ep_t&-\ep\tr [aD^2p^\ep]-b(x)\cdot Dp^\ep
\\&=h'(r_{x,t})\rho|x|\e^{-kt}\left\{
-k-|x|^{-1}b(x)\cdot \bar x
-\fr{\ep}{|x|^2}\tr[a(I-\bar x\otimes\bar x)]\right\}
\\&\quad -\ep h''(r_{x,t})\rho^2\e^{-2kt}\tr[a\bar x\otimes\bar x]
\\&\geq h'(r_{x,t})\rho|x|\e^{-kt}\left\{
-k+b_0
-\fr{n\ep}{\gth_0|x|^2}\right\}
-\ep \|h''\|_{L^\infty}\rho^2\e^{-2kt}\gth_0^{-1}. 
\end{aligned}
\]

 Observe that 
\beq\lab{eq:1.1}
\fr 12\leq r_{x,t}\leq 1 
\ \ \ \text{ if and only if } \ \ 
\fr 12 R\sqrt \ep \e^{kt}\leq|x|\leq R\sqrt\ep\e^{kt},
\eeq
and
\[
h'(r_{x,t})\fr{1}{|x|^2}
\leq h'(r_{x,t})\fr{4\e^{-2kt}}{R^2\ep}
\leq h'(r_{x,t})\fr{4}{R^2\ep}. 
\]

Using the observations above and \eqref{borigin} 
and recalling the choices of the constants and  
that $a\in\bbS^n(\gth_0)$ is arbitrary, we get
\[
\begin{aligned}
p^\ep_t &-\ep P^+(D^2p^\ep)-b(x)\cdot Dp^\ep
\\&\geq 
h'(r_{x,t})\rho|x|\e^{-kt}\left\{
-k+b_0
-\fr{4n}{\gth_0 R^2}\right\}
-\|h''\|_{L^\infty}\fr{\e^{-2kt}}{\gth_0R^2}
\geq 
-\|h''\|_{L^\infty}\fr{\e^{-2kt}}{\gth_0R^2}.
\end{aligned}
\]

Thus, noting that, for all $t>0$,  
\[
p^\ep_t(0,t)-\ep P^+(D^2p^\ep(0,t))-b(0)\cdot Dp^\ep(0,t)=0\quad 
\]
we conclude that
\[
p^\ep_t -\ep P^+(D^2p^\ep)-b(x)\cdot Dp^\ep\geq 
-\|h''\|_{L^\infty}\fr{\e^{-2kt}}{\gth_0R^2} \ \ \text{in} \ \ B_{r_0}\tim(0,\,\infty),
\]
and, hence,  $q^\ep$ is a supersolution of \eqref{eq:Pucci} in $B_{r_0}\tim(0,\,\infty)$.
\smallskip

 Finally, we observe that, 
if $\,0\leq t\leq \tau\,$ and $\,x\in\pl B_{r}$, then 
\[
\fr{|x|\e^{-kt}}{\sep R}\geq \fr{r\e^{-k\tau}}{\sep R}=1
\quad \text{ and } 
\quad
q^\ep(x,t)\geq p^\ep(x,t)=1, 
\]
and, if $\,x \in B_{r/2}$, then 
\[
\fr{|x|\e^{-k\tau}}{\sep R}\leq \fr{r\e^{-k\tau}}{2\sep R}=\fr 12
\quad \text{ and }\quad 
q^\ep\leq p^\ep+\fr{\|h''\|_{L^\infty}}{b_0\gth_0 R^2}
=\fr{\|h''\|_{L^\infty}}{b_0\gth_0 R^2}. 
\]

Moreover, 
\[
q^\ep(x,0)=p^\ep(x,0)=h(|x|/(\sep R))
\geq 
\bcases
0 \quad &\text{ for all }\ x\in B_r,\\[3pt]
1 & \text{ for all } \ x\in B_r\setminus B_{\sep R}. 
\ecases \qedhere
\]
\end{proof}


\subsection*{An application of the Harnack inequality}\lab{sec:harnack}  We use a consequence of the Harnack inequality to obtain an a priori bound  for the 
oscillations  of the $u^\ep$'s, which are uniform in $\ep$ and $t$ up to $\infty$. 
\smallskip

 If $u^\ep\in C^{2,1}(Q)$ is a solution of \eqref{eq:l-pde}, then  
\[
v^\ep(y,t):=u^\ep(\sep y,t) \ \ \ \text{ for } \ (y,t)\in B_{r_0/\sep}\tim[0,\infty),  
\]
satisfies 
\begin{equation}\label{eq:2.1}
v_t^\ep =\tr[a^\ep(\sep y,t)D^2v^{\ep}] + \fr{b(\sep y)}{\sep}\cdot Dv^\ep   \  \text{ in } \ 
B_{r_0/\sep }\tim(0,\,\infty).
\end{equation}

 It also follows from \eqref{blip} 
that there exists   
$L_b>0$ such that
\[
|b(x)|\leq L_b|x| \ \ \ \text{ for all } \ x\in B_{r_0},
\]
and, hence,  
\beq\label{takis3}
\fr{|b(\sep y)|}{\sep }\leq L_b|y| 
\ \ \ \text{ for all } \ y\in B_{r_0/\sep}.  
\end{equation}

 Next we recall the following consequence of the Harnack inequality from  Krylov \cite{Kr1987}*{Theorem 4.2.1}.

\begin{prop}\label{thm:harnack}
Assume \eqref{takis1} and \eqref{takis3},  fix $R\in(0,\,2]$, $(z,\tau)\in\R^n\tim(0,\infty)$ such that 
$
B_R(z)\subset B_{r_0/\sep}) \ \text{ and } \  \tau> 2R^2,
$
and let 
$w\in C^{2,1}(B_{R}(z)\tim(\tau-2R^2,\,\tau))$ be a nonnegative solution 
of  \eqref{eq:2.1} in 
$B_R(z)\tim(\tau-2R^2,\,\tau)$.  
There exists a constant $C=C(R,\gth_0,L_b, n)>1$ such that
\[
w(z,\tau-R^2)\leq C \inf_{y\in B_{R/2}(z)}w(y,\tau). 
\]
\end{prop}

 We use now Proposition~\ref{thm:harnack} to obtain the following improvement 
of oscillation-type result for solutions to \eqref{eq:l-pde}. 

\begin{cor} \lab{cor:harnack} Assume  \eqref{takis1} and \eqref{takis3}  and, for  $\ep\in(0, 1)$, let  $u^\ep\in C(\lbar Q)\cap C^{2,1}(Q)$ be a 
solution of \eqref{eq:l-pde} in $Q$. Fix  $m\in\N$ and $T>0$ and assume   
that $\, (m+2)\sep\leq r_0, \  T>4(m+1)\,$ and 
\begin{equation}\label{eq:2.2}\bcases
u^\ep(0,t)\leq 0 \ \ \ &\text{ for all } \ t\in (0,\,T),
\\[1mm] u^\ep(x,t)\leq 1  \ \ \ &\text{ for all } \ (x,t)\in B_{(m+2)\sep}
\tim(0,\,T).
\ecases
\end{equation}
There exists a constant $\eta=\eta(m,\gth_0,L_b,n)\in(0,\,1)$ such that 
\[
u^\ep\leq \eta \  \text{ in  } \  B_{m\sep }\tim (4(m+1),\,T). 
\]
\end{cor}

\bpr Noting that the function $v^\ep(y,t)=u^\ep(\sep y,t)$ 
is defined on $B_{m+2}\tim(0,\,T)$, we set 
\[
w(x,t)=1-v^\ep(x,t) \ \ \ \text{ for } \ (x,t)\in B_{m+2}\tim(0,\,T).
\]

Observe that $w$ is a solution 
of \eqref{eq:2.1} in $B_{m+2}\tim(0,\,T)$ and, by \eqref{eq:2.2}, that $w$ is a nonnegative function on $B_{m+2}\tim(0,\,T)$ and satisfies 
\[
w(0,t)\geq 1 \ \ \ \text{ for all } \ t\in(0,\,T).
\]

 Let $(x,t)\in B_{m}\tim (4(m+1),\,T)$ and 
choose a finite sequence of balls $B_1(x_1),...,B_1(x_m)\subset B_m$    
so that $x_1=0$, $x\in B_1(x_m)$ and, if $1\leq i<m$, then 
$B_1(x_{i+1})\cap B_1(x_i)\not=\emptyset$. 
Applying  Proposition~\ref{thm:harnack} with $R=2$ yields, for some  $C=C(\gth_0, L_b, n)>1$,
\[ 
w(0,t-4m) 
\leq C \inf_{y\in B_1(x_1)}w(y,t-4(m-1)),
\]
and, hence, if $m=1$,  we have 
\[
w(0,t-4m)\leq C^mw(x,t), 
\]

while, if $m>1$, repeating  the argument above we obtain 
\[\begin{aligned}
w(0,t-4m) 
&\,\leq C w(x_2,t-4(m-1))
\leq C^2\inf_{y\in B_1(x_2)}w(y,t-4(m-2))
\\&\,\leq \cdots \leq C^m \inf_{y\in B_1(x_m)}w(y,t)
\leq C^m w(x,t).
\end{aligned}\]

Thus, 
we have $\,w(0,t-4m)\leq C^mw(x,t)$, and,  
since $\,w(0,t-4m)\geq 1\,$ by \eqref{eq:2.2}, we get 
\[
1\leq C^m (1-v^\ep(x,t)),
\]
which yields
\[
v^\ep(x,t)\leq 1-\fr 1{C^m}, 
\]
and, hence, with   $\eta=1-1/C^m$,
\[
u^\ep(x,t)\leq \eta \quad\text{ for all }(x,t)\in B_{m\sep}\tim(4(m+1),\,T). \qedhere
\]
\epr

\subsection*{The asymptotic constancy} \lab{sec:asymptotic}\lab{sec:origin}
Let $\varPi$ be a relatively open, possibly empty, subset of $\pl\gO$, 
set $\gO^\varPi:=\gO\cup\varPi$, and, 
for any $\gd>0$,
\[
\gO_\gd:=\{x\in\lbar\gO\mid \dist(x,\pl\gO)>\gd\}
\ \ \text{ and } \ \ 
\gO^\Pi_\gd:=\{x\in\lbar\gO\mid\dist(x,\pl\gO\setminus\Pi)>\gd\}. 
\] 

 The next result is the first indication of what we call asymptotic 
constancy, which is a straightforward 
generalization of \cite{IS2015}*{Theorem 14}. 
Roughly it says that, for $\ep $ small, if a solution of \eqref{eq:l-pde} is bounded  and small (say negative) in a small cylinder around 
the positive time axis $l$
and a portion of the parabolic boundary, then it is small (of order $\delta>0$) 
in a large part of $Q$ after some uniform time depending on $\delta$.

\begin{prop}\lab{thm:spreading}
Assume \eqref{omega}, \eqref{blip}, \eqref{g-astable}, \eqref{b-inward}, \eqref{borigin} and \eqref{takis1} and fix $\gd\in (0,\,r_0)$. There exist $T_\gd>0$ and $\ep_0\in(0,\,1)$, 
which depend only on $\gd$, $\gth_0$, $b$, $\Pi$ and $\gO$, such that, if,    
for  $\ep\in(0,\,\ep_0)$,  $u^\ep\in C(\lbar Q) \cap C^{2,1}( Q)$ is a solution of \eqref{eq:l-pde} 
and satisfies, for some $T(\ep)\in(T_\gd,\,\infty]$,
\[
u^\ep\leq 1 \   \text{ in } \ \gO\,\tim\,[0,\,T(\ep)) \ \ \text{and} \ \  
u^\ep\leq 0  \ \text{ in } \  (B_\gd\cup\Pi) \tim [0,\,T(\ep)),
\]
then
\[
u^\ep(x,t)\leq \gd \ \ \  \text{ for all } \ (x,t)\in\gO^\Pi_\gd\tim[T_\gd,\,T(\ep)). 
\] 
\end{prop}
\smallskip

For the proof of Proposition~\ref{thm:spreading} it is necessary to first describe some preliminary facts that are consequence of the asymptotic stability property of the vector field $b$.   
\smallskip

 We fix $\delta>0$ and set
\[
\tau(x):=\sup\{t\geq 0\mid X(t,x)\not\in B_{\gd}\}\quad 
\text{ for }\ x\in\lbar\gO,  
\]
where $X(t)=X(t,x)$ is the solution of 
\[
\dot X(t;x)=b(X(t;x))\quad\text{ and }\quad X(0;x)=x. 
\]

Since $\gO$ is bounded and the origin is a 
globally asymptotically stable point of $b$,
it is immediate that, if  
\beq\lab{eq:T_gd}
T_\gd:=\sup_{x\in \lbar\gO}\tau(x),
\eeq
then 
\beq
0<T_\delta <\infty  \ 
\text{ and}  \  X(t,x)\in B_{\gd} \ \text{for all } \ (x,t) \in \lbar\gO \times 
[T_\gd,\,\infty).
\eeq

 We consider the transport problem
\beq\lab{eq:transport}
\bcases
U_t\leq b\cdot DU \ \  \ &\text{in} \  \gO\tim(0,\,T_\gd],\wcr
\min\{U_t - b\cdot DU,\, U\}\leq 0
&\text{ on } \  \Pi\tim(0,\, T_\gd],\wcr 
U\leq 0&\text{ in }  B_{\gd}  \times \{0\};
\ecases 
\eeq
the first inequality in \eqref{eq:transport} should be understood in the viscosity 
subsolution sense while the second is a viscosity 
interpretation of the Dirichlet 
condition, $U\leq 0$, on $\Pi$ (see  \cite{Is1989}).   

\begin{lem} \lab{lem:transport}
Assume  
\eqref{omega}, \eqref{blip}, \eqref{g-astable} and \eqref{b-inward}. 
If $U\in\USC(\lbar\gO\tim[0,\,T_\gd])$ is a subsolution of \eqref{eq:transport}, 
then $U(x,T_\gd)\leq 0$ for all $x\in\gO^\Pi$. 
\end{lem}

\bpr Fix  $x\in\gO^\Pi$ and, for  $t\in[0,\,T_\gd]$, set 
\[
u(t)=U(X(T_\gd-t,x),t).
\] 

It is a standard observation (see Lemma~\ref{lem:A-transport} in Appendix~A)  
that $u\in\USC([0,\,T_\gd])$ is a subsolution,  if $x\in\gO$,  of 
\beq\lab{eq:visco1}
u'\leq 0 \quad\text{ in }\ (0,\,T_\gd],
\eeq
and, if $x\in\Pi$, of 
\beq\lab{eq:visco2}\bcases
u'\leq 0 \qquad \text{ in }\ (0,\,T_\gd),&\wcr
u'\leq 0 \ \ \text{ or } \ \ u\leq 0 \quad\text{ on  } \ \{T_\gd\}. 
\ecases
\eeq

 Suppose that
$\max_{[0,\,T_\gd]}u>0$. Since $X(T_\gd,x)\in B_\gd$ 
and $u(0)=U(X(T_\gd,x),0)\leq 0$, there must exist $\ga>0$ and $\tau\in(0,\,T_\gd]$ \ such that the function $[0,\,T_\gd] \owns t \to u(t)-\ga t$  attains its 
maximum on $[0,\,T_\gd]$ at $\tau$.  
In view of \eqref{eq:visco1}, if $x\in\gO$, then 
$\ga\leq 0$, which is a contradiction. If $x\in\Pi$, then 
either $\ga\leq 0$ or $\tau=T_\gd$ and $u(T_\gd)\leq 0$, which is again a  contradiction. 
Thus, we conclude that $u\leq 0$ on $[0,\,T_\gd]$. In particular, 
$u(T_\gd)\leq 0$, which shows that $U(x,T_\gd)\leq 0$ for all $x\in\gO^\Pi$.  
\epr

 We proceed with the proof of Proposition~\ref{thm:spreading}.

\bpr[Proof of Proposition~\ref{thm:spreading}]  Let $T_\gd>0$ be the number 
defined by \eqref{eq:T_gd}.  
For any $\ep\in(0, 1)$, let $\cV_\ep$ denote the set of all (viscosity) 
subsolutions $v\in \USC(\lbar\gO\tim[0,\,T_\gd])$ of
\eqref{eq:Pucci} such that  
\beq
v \leq 1 \  \text{ on } \  \lbar \gO \tim [0,\,T_\gd] \ \ \text{and} \ \ v\leq 0 \  \text{ on } \  (B_\gd\cup\Pi) \tim [0,\,T_\gd],
\eeq 
and note that $\cV_\ep$, which is clearly nonempty, 
 depends only on $\gd$, $T_\gd$, $\gth_0$, $b$, $\Pi$ 
and $\gO$.  
\smallskip

 It turns out that $\cV_\ep$ has a maximum element. Indeed, for $(x,t) \in \lbar\gO\tim[0,\,T_\gd]$, set
\[
v^\ep(x,t):=\sup\{v(x,t)\,:\, v\in\cV_\ep\}
\] 
and consider its upper semicontinuous envelope 
\[
\bar v^\ep(x,t):=\lim_{r\to 0}\sup\{v^\ep(y,s)\,:\, 
(y,s)\in\lbar\gO\tim[0,\,T_\gd], \,
|(y,s)-(x,t)|<r\}.
\] 

Standard arguments from the theory of viscosity solutions   yield that  
$\bar v^\ep\in \cV_\ep$ and, since $0\in\cV_\ep$, 
$\bar v^\ep \geq 0$ on  $\lbar\gO\tim[0,\,T_\gd]$. 
\smallskip

 Let $U \in \USC(\lbar \gO\tim[0,\, T_\gd])$ be  
the half-relaxed upper limit of $\bar v^\ep$, that is, for   $(x,t) \in \lbar\gO\tim[0,\,T_\gd]$, 
\[
U(x,t):=\limsup_{\ep\to 0}{\!}^* \bar v^\ep(x,t); 
\]
 we refer to Crandall, Ishii and Lions \cite{CIL1992} for more discussion about the half relaxed upper and lower limits.
\smallskip
 
It follows from Lemma~\ref{lem:transport} that 
$\,U(x,T_\gd)\leq 0  \  \text{ for all } \ x\in \gO^\Pi,$
and, hence, in view of the uniformity encoded in the definition of $U$,  there exists a constant $\ep_0\in(0,\,1)$,   
depending only on $\gd$, $\gth_0$, $b$, $\Pi$ and $\gO$,  
such that, for all  $\ep\in(0,\,\ep_0)$,
\[
v^\ep(\cdot, T_\gd)\leq \gd \ \text{ on } \  \gO^\Pi_\gd. 
\]

Finally, since, for each $\ep$, the function   
$\,\lbar\gO\tim[0,\,T_\gd]\ni (x,t)\mapsto u^\ep(x,s+t),$ 
with $0\leq s<T(\ep)-T_\gd$,  belongs to $\cV_\ep$, 
it follows that,  if $s\in[0,\,T(\ep)-T_\gd)$, then 
\[
u^\ep(x,s+T_\gd)\leq v^\ep(x,T_\gd)\leq \gd \ \ \ \text{ for all } \ x\in \gO^\Pi_\gd \ \text{ and } \ \ep\in(0,\ep_0], 
\]
and, thus, 
\[
u^\ep(x,t)\leq \gd \ \  \ \text{ for all }(x,t)\in\gO^\Pi_\gd\tim[T_\gd,\,T(\ep))
\ \text{ and } \ \ep\in(0,\,\ep_0].  \qedhere
\]
\end{proof}
\smallskip

  Next we use Corollary~\ref{cor:harnack} and the previous proposition to obtain a refinement. Here we assume an upper bound, say $1$,  only in a cylindrical neighborhood of  
the positive time axis $l$ 
and show that, if, in addition, the solutions are small, say less than $0$ on 
the half line $l$, then they are small, say less than $\delta$,  
after a time, of order $|\log\ep|$, in  
a small cylindrical neighborhood of $l$. We remark 
that a time period 
of order $|\log\ep|$ is ``very short'' in the logarithmic scale of time,
that is, as $\ep \to 0$, 
if $\exp(\gl_\ep/\ep)=O(|\log\ep|)$, then $\gl_\ep \to 0$.    
  
\begin{prop}\lab{thm:con-origin1} Assume \eqref{omega}, \eqref{blip}, \eqref{g-astable}, \eqref{borigin} and \eqref{takis1}. For any $\gd>0$, there exist 
 $\ep_0 \in(0,\,1)$
and a family $\{\tau(\ep)\}_{0<\ep\leq\ep_0}\subset (0,\,\infty)$, both depending on $r_0$, $\gth_0$, 
$b$,  $\gd$ and $n$,  and $\gamma\in(0,\,1)$, 
such that, if, for  $\ep\in(0,\,\ep_0]$,  
 $u^\ep$ is a solution of \eqref{eq:l-pde} with the property that,   for some $T(\ep)\in(\tau(\ep),\,\infty]$,
\beq \lab{eq:Upton}
u^\ep \leq 1 \  \text{ in  } 
B_{r_0}\tim (0, T(\ep)) \ \  \text{and} \ \ 
u^\ep(0,t)\leq 0  \ \text{ for all } \  t\in(0,\,T(\ep)), 
\eeq
then 
\[
u^\ep \leq \gd \ \  \text{ in  } \ \  B_{\gamma r_0}\tim(\tau(\ep),\,T(\ep)). 
\]
Moreover, there exists a constant $C>0$, which depends  on $r_0$, $\gth_0$, 
$b$, $\gd$ and $n$, such that
\[
\tau(\ep)\leq C(|\log\ep|+1) \ \ \ \text{ for all } \ \ep\in(0,\,\ep_0].
\] 
\end{prop}
\smallskip

 Although it appears similar, Proposition~\ref{thm:con-origin1} is actually very different from  \cite{IS2015}*{Theorem 13}.
Indeed the second condition in \eqref{eq:Upton} on the solutions  
is required only at the origin, while  
in \cite{IS2015}*{Theorem 13} it is assumed  on a neighborhood of the origin. 
This refinement, which  is important for the proofs of 
Propositions~\ref{thm:A}, \ref{thm:B} and \ref{thm:C},   
depends technically on the barrier functions $q^\ep$ in Lemma~\ref{lem:q} and 
the Harnack inequality (Proposition~\ref{thm:harnack}).

\bpr[Proof of Proposition~\ref{thm:con-origin1}]  
To simplify the argument, we assume that $T(\ep)=\infty$ since the general case 
can be treated similarly. 
\smallskip

 Fix $\gd>0$, choose $h\in C^2([0,\,\infty))$ satisfying \eqref{eq:h} and 
$m=m(\gth_0,n,\|h''\|_{L^\infty})\in\N$ 
such that
\[
\fr{\|h''\|_{L^\infty}}{b_0\gth_0 m^2}\ \leq\ \fr 12 
\quad\text{ and }\quad m\geq \fr{2\sqrt{2n}}{\sqrt{b_0\gth_0}}, 
\]
 let $\eta=\eta(\gth_0,L_b,n)\in(0,\,1)$ be 
the constant in Corollary \ref{cor:harnack},
set $\tau_0=4(m+1)$ and fix $\ep_1=\ep_1(r_0,m)\in(0,\,1)$ so that 
\[
(m+2)\sqrt{\ep_1}\leq r_0. 
\]

 Then, for any $\ep\in(0,\,\ep_1]$,
Corollary \ref{cor:harnack} gives 
\[
u^\ep(x,t)\leq \eta \ \ \ \text{ for all } \ (x,t)\in B_{m\sep }
\tim(\tau_0,\,\infty). 
\]

Define
\[
v^\ep:=(1-\eta)^{-1}\left(u^\ep-\eta\right)
 \  \text{ in }  \ \gO\tim[0,\,\infty),
\]
and note that $v^\ep$ is a solution of \eqref{eq:l-pde}, and, moreover, 
\[
v^\ep \leq 1 \  \text{ in } \  B_{r_0}\tim (0,\infty) \ \  \text{ and} \ \  v^\ep\leq 0 \ \text{ on } \  \lbar B_{m\sep}\tim [\tau_0,\,\infty).
\]

 Let $q^\ep$ be given by \eqref{eq:q} with $R$ and $r$ replaced by $m$ 
and $r_0$ respectively. 
 It follows from  Lemma~\ref{lem:q} and the comparison principle that, for any fixed $s\geq \tau_0$, 
\[
v^\ep(\cdot,s+\cdot)\leq q^\ep   \ \text{ in } \  B_{r_0}\tim[0,\,\tau_1],
\] 
where $\tau_1=\tau_1(\ep)>0$ is  given by 
\[
\fr{\gth_0\tau_1}{2}=\log\left(\fr{r_0}{m\sep}\right).
\]

Hence, 
\[
v^\ep(\cdot, \cdot +\tau_1)\leq \fr{\|h''\|_{L^\infty}}{b_0\gth_0 m^2}\leq \fr 12 
\  \text{ in } \  B_{r_0/2}\tim [\tau_0,\,\infty),
\]
which, with $T_1(\ep):=\tau_0+\tau_1(\ep)$, can be rewritten as 
\beq\lab{eq:induction1}
u^\ep \leq  \eta+\fr{1-\eta}{2}=\fr{1}{2}(1+\eta) \  \text{ in } \  B_{r_0/2}\tim[T_1(\ep),\,\infty).
\eeq

 Next, for $j=2,3,...$, 
we choose $\ep_j\in(0,\,\ep_{j-1})$ so that 
\[
(m+2)\sqrt{\ep_j}\leq \fr{r_0}{2^{j-1}},
\] 
and, for any $\ep\in(0,\,\ep_j)$, select  
$\tau_j=\tau_j(\ep) > \tau_{j-1}(\ep)$ so that 
\[
\fr{\gth_0\tau_j(\ep)}{2}=\log\left(\fr{r_0}{2^{j-1}m\sep}\right),  
\]
and set, for  $ \ep\in(0,\,\ep_j)$, 
\[
T_j(\ep):=T_{j-1}(\ep)+\tau_0+\tau_j(\ep)=j\tau_0+\sum_{i=1}^j\tau_i(\ep). 
\]

 We prove by induction that   
\beq\lab{eq:inductionj}
u^\ep\leq \left(\fr{1+\eta}{2}\right)^j \ \text{ in } \  B_{r_0/2^j}\tim [T_j(\ep),\,\infty). 
\eeq

  Since  \eqref{eq:induction1} yields  that \eqref{eq:inductionj} holds for $j=1$, we 
assume that \eqref{eq:inductionj} is valid for some $j\in\N$, 
set 
\[
w^\ep:=\left(\fr{2}{1+\eta}\right)^{j}u^\ep(\cdot, \cdot+T_j(\ep)) 
\ \text{ in } \  Q, 
\]
observe  that $w^\ep$ is a solution of \eqref{eq:l-pde}, with 
$a^\ep(\cdot, \cdot)$ replaced by $a^\ep(\cdot, \cdot+T_j(\ep))$ and satisfies
$w^\ep(0,t)\leq 0$ for all $t\in[0,\,\infty)$ and $w^\ep\leq 1$ in  $B_{r_0/2^j}\tim[0,\infty)$.
\smallskip

Using Lemma~\ref{lem:q} and Corollary~\ref{cor:harnack} as before, 
with the same $m$ and $\tau_0$, but with 
$u^\ep$, $r_0$ and $\tau_1$ replaced by $w^\ep$, $r_0/2^{j}$ and $\tau_{j+1}$ 
respectively, 
we  obtain 
\[
w^\ep \leq \fr{1+\eta}{2} \  \text{ in } \  B_{r_0/2^{j+1}}\tim(\tau_0+\tau_{j+1}(\ep),\,\infty), 
\] 
which, after been rewritten as 
\[
u^\ep \leq \left(\fr{1+\eta}{2}\right)^{j+1} 
\  \text{ in} \  B_{r_0/2^{j+1}}
\tim[T_{j+1}(\ep),\,\infty),
\]
yields the claim.  
\smallskip

 Finally, selecting $j\in\N$ so that 
\[
\left(\fr{1+\eta}{2}\right)^{j}\leq\gd,
\]
setting $\ep_0=\ep_j$, $\gamma=2^{-j}$ and $\tau(\ep)=T_j(\ep)$, 
and observing that, as $\ep\to 0+$, $\tau(\ep)=O(|\log \ep|)$ we complete 
the proof. 
\epr
\smallskip

 We have by now completed all the technical steps needed for the next theorem, which is a nontrivial refinement  of Proposition~\ref{thm:spreading}. It asserts that bounded solutions 
to \eqref{eq:l-pde},  which are small on the 
positive time axis $l$ and a part of the parabolic boundary, are actually small in almost the whole domain after some time of order 
$|\log \ep|$. This is the mathematical statement of what we called asymptotic constancy.

\begin{thm}\lab{thm:conclusion}  Assume \eqref{omega}, \eqref{blip}, \eqref{g-astable}, \eqref{b-inward}, \eqref{borigin} and \eqref{takis1} and let $\{T(\ep)\}_{\ep\in(0,\,1)}$ be a collection of  
positive numbers. For each $\gd>0$ and $C_0>0$, there exist constants 
$\ep_0\in(0,\,1)$ and $C>0$ such that, 
if,  for $\ep\in(0,\,\ep_0]$, 
$u^\ep\in C^{2,1}(Q)$ is a solution of \eqref{eq:l-pde} satisfying  
\[
u^\ep \leq C_0  \ \text{ in } \ \gO\tim [0,\,T(\ep)) \ \ \text{and} \ \  u^\ep \leq 0 \ \text{ in} \ (\{0\} \cup \Pi) \times [0,\,T(\ep)),
\]
then 
\[
u^\ep(x,t)\leq \gd \ \ \ \text{ for all  $(x,t)\in\gO^\Pi_{\gd}\tim (C|\log\ep|,\,T(\ep))$.} 
\] 
\end{thm}

\bpr 
Proposition~\ref{thm:con-origin1} yields constants 
$\ep_1,\,\gamma\in(0,\,1)$ and $C_1>0$ such  that, 
for all $0<\ep\leq \ep_1$, 
\[
u^\ep\leq \fr{\gd}2 \ \text{ in  } \  B_{\gamma r_0}\tim[C_1|\log \ep|,\, T(\ep)).  
\]

Proposition~\ref{thm:spreading} applied to $v^\ep(x,t):=C_0^{-1}(u^\ep(x,t+C_1|\log\ep|)-\gd)$ instead $u^\ep$ 
implies the existence of $T_\delta$ and $\ep_0 \in (0,\ep_1)$ such that,  for any $\ep\in(0,\,\ep_0)$,  
\beq\lab{assertion2}
v^\ep\leq \fr \gd{2C_0} \ \text{ in  } \ 
\gO^\Pi_\gd\tim[T_\gd,\,T(\ep)-C_1|\log\ep|),
\eeq
which says that, for any $\ep\in(0,\ep_0)$, 
\[
u^\ep \leq \gd \ \text{ in } \ \gO^\Pi_\gd\tim[T_\gd+C_1|\log\ep|,\,T(\ep)),
\]
and the proof is complete.
\epr

 Next we use the last result  to control the difference between values of $u^\ep(\cdot, t)$ and $u^\ep(0,t)$. 

\begin{thm}\lab{thm:conclusion3}  Assume \eqref{omega}, \eqref{blip}, \eqref{g-astable}, \eqref{b-inward}, \eqref{borigin} and \eqref{takis1}. For each  $\gd>0$ and $C_0>0$ there exist
constants $\ep_0\in(0,\,1)$ and $C>0$ such that, if, for $\ep\in(0,\,\ep_0]$,  $u^\ep$ is a solution 
of \eqref{eq:l-pde} satisfying
\[
|u^\ep|\leq C_0 \  \text{ in } \ \gO\tim[0,\,\infty),
\]
then 
\[
|u^\ep(x,t)-u^\ep(0,t)|\leq \gd  \ \ \text{ for all  }\ (x,t)  \in 
\gO_\gd\tim [C|\log\ep|,\,\infty).
\]
\end{thm}

\bpr We double the variables and define the function $v^\ep: \gO\tim\gO\,\tim\,[0,\,\infty) \to \R$ by 
\[
v^\ep(x,y,t):=u^\ep(x,t)-u^\ep(y,t). 
\]

It is standard that  $v^\ep$ solves in $\gO\tim\gO\tim (0,\,\infty)$ the doubled equation
\[\begin{aligned}
v^\ep_t&\,
=\ep\tr[a^\ep(x,t)D_x^2v^\ep]
+\ep\tr[a^\ep(y,t)D_y^2v^\ep]
+b(x)\cdot D_xv^\ep +b(y)\cdot D_yv^\ep
\\&\,=\ep\tr[A^\ep(x,y,t)D^2v^\ep]+B(x,y)\cdot Dv^\ep,
\end{aligned}\]
where 
\[
B(x,y):=(b(x),b(y)) \quad
\text{ and }\quad
A^\ep(x,y,t):=\begin{pmatrix}
a^\ep(x,t)&0\\ 0&a^\ep(y,t)
\end{pmatrix}. 
\]

 The conclusion follows if we apply Theorem~\ref{thm:conclusion}, with $\Pi=\emptyset$, to $\pm v^\ep$, since
 $v^\ep(0,0,t)=0$ for all $t\geq 0$ and $|v^\ep|\leq 2C_0$ in 
$\gO\tim\gO\,\tim\,[0,\,\infty)$. 
\smallskip

 The only issue is that the boundary of $\gO\tim\gO$ does not have the $C^1$- 
regularity required for the theorem.
\smallskip

 To overcome this difficulty, we only need to approximate $\gO\tim\gO$ by 
smaller $C^1$-domains, that is, for fixed $\gd>0$, we choose 
a $C^1$-domain $W\subset \R^{2n}$ so that 
\[
\gO_\gd\tim\gO_\gd \subset W_{\gd/2}\subset W\subset \gO\tim\gO,
\]
where $W_{\gd/2}:=\{(x,y)\in W\mid \dist((x,y),\pl W)>\gd/2\}$, and 
\[
B(x,y)\cdot N(x,y)<0 \quad\text{ for all } (x,y)\in\pl W,
\]
where $N(x,y)$ denotes the outward unit normal vector at $(x,y)\in \pl W$.  
\epr

\section{Quasi-potentials}\lab{sec:quasi-p}

 We establish here an important continuity property under perturbations for the minimum and the $\argmin$ map of the quasi-potentials we introduced earlier in the introduction.
\smallskip
   
  We begin with some notation and the introduction of several auxiliary quantities needed to define the perturbations. To this end, we  fix $\gb_0\in
I_g$,  
define $H_0\in C(\lbar\gO\tim\R^n)$ by 
\[H_0(x,p)=a(x,\gb_0)p\cdot p+b(x)\cdot p,
\] 
choose some 
$\gd_0>0$, and, 
for $\gd\in(0,\,\gd_0)$,
\[
\gth(\gd):=\max\{
|(a(x,c)-a(x,\gb_0))\xi\cdot\xi|\mid x\in\lbar\gO,\,
\xi\in\R^n,\, |\xi|\leq 1,\, c\in[\gb_0-\gd,\,\gb_0+\gd]\}.
\]

The continuity of $a(x,c)$ (recall \eqref{ellipticity}) yields  $\lim_{\gd\to 0}\gth(\gd)=0$, and, hence, 
selecting $\gd_0>0$ sufficiently small, we assume henceforth
that 
\[
\gth(\gd)\leq \gth_0/2\quad \text{ for all }\gd\in(0,\,\gd_0).
\] 

 We define $a_\gd^\pm \in C(\lbar\gO,\,\bbS^n)$ and 
$H_\gd^\pm \in C(\lbar\gO\tim\R^n)$  respectively by
\[
a_\gd^\pm(x):=a(x,\gb_0)\pm\gth(\gd)I \ \ \text{and} \ \ 
 H_\gd^\pm(x,p):=a_\gd^\pm(x)p\cdot p+b(x)\cdot p,
\]
and note that, for all $(x,c)\in\lbar\gO\tim[\gb_0-\gd,\,\gb_0+\gd]$, 
\[
(\gth_0/2) I\leq a_\gd^-(x)
\leq a(x,c)\leq a_\gd^+(x)\leq (\gth_0^{-1}+\gth_0/2)I.
\]

 We choose $\chi_\gd \in C(\R^n; [0,1])$ 
such that 
\[
\chi_\gd=1 \ \ \text{ in  } \ \ x\in\gO_{\gd}\quad
\text{ and }\quad
\chi_\gd=0 \ \  \text{ in  } \ \ \R^n\setminus \gO_{\gd/2},
\]
and define $\cH_\gd^\pm \in C(\lbar\gO\tim\R^n)$ by
\[\begin{aligned}
\cH_\gd^+(x,p)&\,=\chi_\gd(x)H_\gd^+(x,p)+
(1-\chi_\gd(x))(\gth_0^{-1}|p|^2+b(x)\cdot p),
\\[1mm]
\cH_\gd^-(x,p)&\,=\chi_\gd(x)H_\gd^-(x,p)+
(1-\chi_\gd(x))(\gth_0|p|^2+b(x)\cdot p),
\end{aligned}
\]
and note that, for all \  
$(x,c)\in \gO_{\gd/2}\tim[\gb_0-\gd,\,\gb_0+\gd]\,\cup\, 
(\gO\setminus \gO_{\gd/2})\tim\R$ 
\ and \ $p\in\R^n$, 
\[
\cH_\gd^-(x,p)\leq 
a(x,c)p\cdot p+b(x)\cdot p
\leq\cH_\gd^+(x,p). 
\]

We also have
\[
\cH_\gd^\pm(x,p)=H_\gd^\pm(x,p)   \ \text{ for all $(x,p)\in \gO_{\gd}\tim\R^n$},
\]
while, for all $(x,p)\in (\lbar \gO \setminus  \gO_{\gd/2})\tim\R^n$,  
\[ 
\cH_\gd^+(x,p)=\gth_0^{-1}|p|^2+b(x)\cdot p \ \  \text{ and } \ \   
\cH_\gd^-(x,p)=\gth_0|p|^2+b(x)\cdot p.
\]

If we set 
\[
\ga_\gd^+(x)=\chi_\gd(x)a_\gd^+(x)+(1-\chi_\gd(x))\gth_0^{-1}I 
\quad\text{ and }\quad 
\ga_\gd^-(x)=\chi_\gd(x)a_\gd^-(x)+(1-\chi_\gd(x))\gth_0I, 
\]
then, for all $(x,p)\in \lbar\gO\tim\R^n$, 
\[
\cH_\gd^\pm(x,p)=\ga_\gd^\pm(x)p\cdot p+b(x)\cdot p. 
\]

 Let $V_0$ and  $V_\gd^\pm$ 
be respectively 
the maximal subsolutions of 
\beq\lab{eq:quasi2-1}
\bcases
H_0(x,Du)=0 \ \text{ in } \ \gO, &\\
u(0)=0,&
\ecases
\eeq
and 
\beq\lab{eq:quasi2-2}
\bcases
\cH_\gd^\pm(x,Du)=0\quad\text{ in }\gO, &\\
u(0)=0.&
\ecases
\eeq 

We note by \cite{IS2015}*{Corollary 5} that $V_\gd^\pm(x)>0$ and $V_0(x)>0$ 
for all $x\in\lbar\gO\setminus\{0\}$. 
 Since $\cH_\gd^-\leq H_0\leq \cH_\gd^+$ on $\gO\tim\R^n$, it is clear that 
\beq\label{takis5}
V_\gd^+\leq V_0\leq V_\gd^- \ \text{ on } \ \lbar\gO. 
\eeq

 We set 
\[M_0:=\min_{\pl\gO} V_0, \ \  
\gG_0:=\argmin(V_0|\pl\gO), \ \  M_\gd^\pm:=\min_{\pl\gO}V_\gd^\pm, \ \ 
\gG_\gd^\pm:=\argmin(V_\gd^\pm|\pl\gO),
\]
and  note that 
\[M_\gd^+\leq M_0\leq M_\gd^-.\]

 The following result is about the continuity of $M_\gd^\pm$ and $\gG_\gd^\pm$ with respect to $\delta$.

\begin{prop} \lab{thm:quasi2-1} Assume  \eqref{omega}, \eqref{ellipticity}, \eqref{blip}, \eqref{g-astable} and \eqref{b-inward}. Then 
\beq\lab{eq:quasi2-4}
\lim_{\gd\to 0+}M_\gd^+=
\lim_{\gd\to 0+}M_\gd^-
=M_0
\eeq
and 
\beq\lab{eq:quasi2-5}
\limsup_{\gd\to 0+}\gG_\gd^+ \cup \limsup_{\gd\to 0+}\gG_\gd^-\subset \gG_0.
\eeq
\end{prop}

 The set  limit in \eqref{eq:quasi2-5} 
is understood in the sense of Kuratowski,  that is, 
for a given $\{\gG_\gd\}_{\gd\in(0,\,\gd_0)}\subset \R^n$, 
\[
\limsup_{\gd\to 0+}\gG_\gd:=\bigcap_{r\in(0,\,\gd_0)}\,
\overline{\bigcup_{\gd\in(0,\,r)}\gG_\gd}
=\{x\in\R^n\mid 
x=\lim_{k\to\infty}x_k, \ x_k\in\gG_{\gd_k},\ \lim_{k\to\infty}\gd_k=0\}.  
\]

Now we prove  Proposition~\ref{thm:quasi2-1}.  

\bpr 
The uniform in $x$ and $\delta$ coercivity 
 of the Hamiltonians $\cH_\gd^\pm$, 
that is the fact that $\cH_\gd^\pm(x,p) \to \infty$ as $|p| \to \infty$ uniformly in $x$ and $\delta$, yields that  the families  $\{V_\gd^\pm\}_{\gd\in(0,\gd_0)}$  
are equi-Lipschitz  continuous on $\lbar\gO$, 
and, since $V_\gd^\pm(0)=0$, 
relatively compact in $C(\lbar\gO)$. 
\smallskip

 To prove \eqref{eq:quasi2-4} and \eqref{eq:quasi2-5}, it is enough to 
show that, if $\{\gd_j\}_{j \in \N}\subset(0,\,\gd_0)$ is such 
 that both $\{V_{\gd_j}^\pm\}_{j\in\N}$ 
converge in $C(\lbar\gO)$ to some $V_0^\pm\in C(\lbar\gO)$, 
that is 
\[
V_0^\pm=\lim_{j\to\infty} V_{\gd_j}^\pm
\quad\text{ uniformly  on } \ \lbar\gO, 
\]
then 
\beq \lab{eq:quasi2-6}
M_0=\min_{\pl\gO}V_0^+=\min_{\pl\gO}V_0^-.
\eeq
and 
\beq\lab{eq:quasi2-7}
\argmin(V_0|\pl\gO)=\argmin(V_0^+|\pl\gO)=\argmin(V_0^-|\pl\gO).
\eeq

 For notational convenience, we set 
\[
M_0^\pm:=\min_{\pl\gO}V_0^\pm \ \ \text{ and } \ \ 
\gG_0^\pm=\argmin(V_0^\pm|\pl\gO).  
\]

 It is well-known (see Lemma~\ref{lem:A-super} 
in the Appendix) that the $V_\gd^\pm$'s satisfy  in the viscosity sense
\[ 
\cH_\gd^\pm(x,DV_\gd^\pm)\geq 0 \ \text{ on }\lbar\gO \ \ \ \text{and} \ \ \ 
\cH_\gd^\pm(x,DV_\gd^\pm)\leq 0 \ \text{ in }\gO,
\]
that is, the $V_\gd^\pm$'s  are 
solutions of the state-constraints problems 
\[
\cH_\gd^\pm(x,DV_\gd^\pm)=0 \quad\text{ in }\gO.
\] 

 By the stability of viscosity properties,  the $V_0^\pm$'s  satisfy 
\[
H_0(x,DV_0^\pm(x))\leq 0 \ \text{ in } \ \gO \ \ \text{ and } \ \
H_{\gth_0}^+(x,DV_0^+(x))\geq 0 \quad\text{ on }\lbar\gO,
\]
where 
\[
H_{\gth_0}^+(x,p):=
\bcases
H_0(x,p)&\text{ for }(x,p)\in\gO\tim\R^n,\\[1mm]
\gth_0^{-1}|p|^2+b(x)\cdot p&\text{ for }(x,p)\in\pl\gO\tim\R^n.
\ecases
\]   

Here we used that 
\[
\limsup_{\gd\to 0}{\!}^*\, \cH_\gd^\pm(x,p)=
\liminf_{\gd\to 0}{\!}_*\,\cH_\gd^\pm(x,p)=H_0(x,p) \ \  
\text{ for all $(x,p)\in \gO\tim\R^n$},
\]
and
\[
\limsup_{\gd\to 0}{\!}^*\,\cH_\gd^+(x,p)=
H_{\gth_0}^+(x,p)  \ \   
\text{ for all $(x,p)\in\lbar\gO\tim\R^n$}.
\]

The maximality of $V_0$  implies  that 
$V_0^-\leq V_0  \ \text{ on} \ \varOmega$ and, since, in view of \eqref{takis5}, $V_0 \leq V_0^- \ \text{in} \ \overline \varOmega$, we have $V_0^-= V_0$, which, 
obviously gives   
\beq\lab{eq:minus}
M_0=M_0^- \quad\text{ and }\quad 
\gG_0^-=\gG_0.
\eeq

   The argument for  $M_0^+$ and $\gG_0^+$ is slightly more complicated. 
\smallskip

Since \eqref{takis5} yields $V_0^+\leq V_0$, it is immediate 
that  
\[M_0^+ \leq M_0.\]

Next we show that
\beq\lab{eq:quasi2-8}
\min\{V_0,\,M_0\}\leq V_0^+ \ \text{ in } \  \lbar\gO,
\eeq
which, together  the previous inequality, gives
\beq\lab{eq:quasi2-10}
M_0^+=M_0 \ \ \text{ and } \ \ \gG_0\subset\gG_0^+.
\eeq

 We proceed with the proof of \eqref{eq:quasi2-8}. Fix  $l\in(0,\,M_0)$, choose $\gamma_1\in (0,\,\gd_0)$ so that
\[
V_0>l \ \text{ on } \ \lbar\gO\setminus\gO_{\gamma_1},
\]
fix $\mu\in(0,\,1)$ sufficiently close to $1$ so that
\[
\mu V_0>l \ \text{ on} \ \lbar\gO\setminus\gO_{\gamma_1},
\] 
and choose $\gamma_2\in (0,\,\gamma_1)$ so that
\[
\mu(a(x,\gb_0)+\gth(\gd)I)\leq a(x,\gb_0)\quad \text{ for all }\ x\in\lbar\gO
\ \text{ and }\ \gd\in(0,\,\gamma_2). 
\]

Observe that, if $u_\mu(x):=\mu V_0(x)$, then, for all $\gd\in(0,\,\gamma_2)$, 
\[
u_\mu>l \ \text{ in  } \ \lbar\gO\setminus \gO_\gd,
\]
and, for all $\gd\in(0,\,\gamma_2)$, in the viscosity sense,
\[\begin{aligned}
H_\gd^+(x,Du_\mu)
&\,=\mu(\mu(a(x,\gb_0)+\gth(\gd)I)DV_0\cdot DV_0+b(x)\cdot DV_0)
\\&\,\leq \mu(a(x,\gb_0)DV_0\cdot DV_0+b\cdot DV_0)\leq \mu H_0(x,DV_0)\leq 0
\quad\text{ in }\ \gO. \end{aligned}\]

 Now set $u_\mu^l:=\min\{u_\mu,\,l\}$
and note that the convexity of $H_\gd^+(x,p)$ in $p$ yields 
that, if $\gd\in(0,\,\gamma_2)$, then  
\[
\cH_\gd^+(x,Du_\mu^l)=
H_\gd^+(x,Du_\mu^l)\leq 0 \quad\text{ in }\ \gO_{\gd}.
\]

Also, if $\gd\in(0,\,\gamma_2)$, then, 
since $u_\mu^l(x)=l$ in an open neighborhood $N_\gd\subset\gO$ 
of $\gO\setminus\gO_{\gd}$, 
\[
\cH_\gd(x,Du_\mu^l)=0 \quad\text{ in }\ N_\gd. 
\]

Thus we deduce that, for any $\gd\in(0,\gamma_2)$, 
$u_\mu^l$ is a subsolution of $\cH_\gd^+(x,Du_\mu^l)\leq 0$
in $\gO$, and, hence,  
\ $u_\mu^l\leq V_\gd^+$ in $\lbar\gO\,$  by the maximality of $V_\gd^+$. 
\smallskip
Sending $\gd\to 0$, along the sequence $\{\gd_j\}$, 
$\mu\to 1$ and $l\to M_0$ in this order, we conclude that 
\eqref{eq:quasi2-8} holds. 
\smallskip

Next we show that $\gG_0^+\subset \gG_0$.  Let $z\in \gG_0^+\setminus\gG_0$ and 
observe that, since $V_0(z)>M_0$, there is an open, relatively to $\lbar\gO$, neighborhood $N_z\subset\lbar\gO$, 
such that $V_0>M_0$ in  $N_z$, 
while  \eqref{eq:quasi2-8} gives  $V_0^+\geq M_0$ in  $N_z$. 

\smallskip
 Let $\rho\in C^1(\R^n)$ be a defining function of $\gO$, that is,  
$\gO=\{x\in\R^n\mid\rho(x)<0\}$ and $|D\rho|\not=0$ on  $\pl\gO$, and,  
in particular, $D\rho/|D\rho|=\nu$ on $\pl\gO$. 
\smallskip 

 For any $\ep>0$, $x \mapsto V_0^+(x)-\ep\rho(x)$ achieves 
a minimum at $z$ over $N_z$. Since $H_{\gth_0}^+(x,DV_0^+)\geq 0$ on
$\lbar\gO$, we have 
\[
0\leq H_{\gth_0}^+(z,\ep D\rho(z))=\ep(\ep\gth_0^{-1}|D\rho(z)|^2
+b(z)\cdot D\rho(z)),
\] 
which is a contradiction, in view of the fact that  the right hand side is negative if $\ep$ 
is sufficiently small. 
\smallskip

 It follows that $\gG_0^+\setminus\gG_0=\emptyset$, that is, 
$\gG_0^+\subset \gG_0$, which, together with \eqref{eq:quasi2-10}, proves the claim. 
\epr

\section{Barrier functions} 

We adapt and modify here 
the main argument of building barrier functions of \cite{IS2015}  
to obtain information on the behavior of the solutions $u^\ep$ of \eqref{eq:l-pde}
along the positive time axis $l$, that is on $u^\ep(0,t)$, for a sufficiently long time interval $[0,\,T(\ep))$, 
under the assumption that 
the matrices $a^\ep \in C(\lbar Q_{T(\ep)})$ are bounded 
by $\ga\in C(\lbar Q_{T(\ep)})$ from above or from below.  
\smallskip  

Recall that, for any  $\ga\in C(\lbar\gO,\bbS^n(\gth_0))$,  
$H_\ga\in C(\lbar\gO\tim\R^n)$ is the Hamiltonian given by 
$H_\ga(x,p)=\ga(x) p\cdot p+b(x)\cdot p$,   
$V_\ga\in \Lip(\lbar\gO)$ is the quasi-potential corresponding to $(\ga,b)$, and 
$M_\ga=\min_{\pl\gO}V_\ga$, and set 
\[\gS_\ga:=\{x\in\lbar\gO\mid V_\ga(x)
\leq M_\ga\}, \quad \gG_\ga:=\gS_\ga\cap\pl\gO,
\] 
and, for any $m>0$,  
\[\gS_\ga^m:=\{x\in\lbar\gO\mid V_\ga(x)\leq m\}.
\]

 We consider again \eqref{eq:l-pde} for a family of $a^\ep\in C(\lbar Q,\bbS^n(\gth_0))$ 
and present two results, one stated in the form of an upper bound and the other in the form of a lower bound. 
The upper bound is valid up to $\gl$ smaller than $M_\ga$ in the logarithmic time scale, and the lower bound is valid up to $\infty$, provided $u^\ep$, on the boundary portion $\gG_\ga\tim[0,\,T(\ep))$, is larger than a specified lower bound.

We begin with the former,  which  corresponds to \cite{IS2015}*{Theorem 1 (i)} in its nature. The latter is related 
 to \cite{IS2015}*{Theorem 1(ii)}.

\begin{prop}\lab{thm:short} 
Assume \eqref{B} and fix 
$\ga\in C(\lbar\gO,\bbS^n(\gth_0))$,  
$T(\ep)\in (0,\,\infty]$, $C_0>0$ and $m\in(0,\,M_\ga)$.  
If, for  $a^\ep \in C(\lbar Q_{T(\ep)}; \bbS^n(\gth_0))$  such that 
$a^\ep\leq \ga$ in $Q_{T(\ep)}$,
$u^\ep\in C(\lbar Q_{T(\ep)})\cap C^{2,1}(Q_{T(\ep)})$  is 
a subsolution of \eqref{eq:l-pde} in $Q_{T(\ep)}$ such that 
 \[
u^\ep(\cdot,0)\leq 0 \ \text{ in } \ \gS_\ga^m \ \ \text{and} \ \ 
u^\ep\leq C_0 \  \text{ in }\ Q_{T(\ep)},
\]
then, for any $\gd>0$, there exists $\ep_0\in(0,\,1)$ such that, if  
$\ep\in(0,\,\ep_0)$, then
\[
u^\ep(0,t)\leq \gd\quad\text{ for all }\ t\in[0,\,\exp((m-\gd)/\ep)\wedge T(\ep)).
\]
\end{prop}

The lower bound is stated next. 

\begin{prop}\lab{thm:long} Assume  \eqref{B}, 
fix $\ga\in C(\lbar\gO,\bbS^n(\gth_0))$,  
$T(\ep)\in (0,\,\infty]$, $C_0>0$ 
and $m>M_\ga$. If, for  $a^\ep \in C(\lbar Q_{T(\ep)};\bbS^n(\gth_0))$
such that $a^\ep\geq \ga$ in $Q_{T(\ep)}$, 
$u^\ep\in C(\lbar Q_{T(\ep)})\cap C^{2,1}(Q_{T(\ep)})$  
is a supersolution of \eqref{eq:l-pde} in $Q_{T(\ep)}$ 
such that 
\[u^\ep(\cdot,0)\geq 0 \ \text{ in } \  \gS_\ga^m,  \ \ u^\ep \geq 0 \ \text{ in } \ (\gS_\ga^m\cap\pl\gO)\tim(0,T(\ep)) \ \text{and} \ u^\ep\geq -C_0\ \ \text{ in } Q_{T(\ep)},
\]
then, for any $\gd>0$, there exists $\ep_0\in(0,\,1)$ such that, if 
$\ep\in(0,\,\ep_0)$, then
\[
u^\ep(0,t)\geq -\gd\quad\text{ for all }\ t\in[0,\,T(\ep)).
\]
\end{prop}

The proofs of Propositions~\ref{thm:short} and \ref{thm:long} use the next two lemmata;  for their proof we refer  to \cite{IS2015}.
 
\begin{lem} \lab{lem:v_r} Assume \eqref{B} and fix $\ga\in C(\lbar\gO,\bbS^n(\gth_0))$. For any $r\in (0,\,r_0)$, there exist $v_r\in C^2(\lbar\varOmega)$ and 
$\eta \in (0,1)$ such that 
\beq\label{takis10}
\bcases
H_\ga(x,Dv_r)\leq -\eta \ \text{ in } \ \varOmega\setminus B_r, \ &
\\ H_\ga(x,Dv_r)\leq 1 \ \text{ in } \ B_r, \ &
\\ \|v_r-V_\ga\|_{L^\infty(\varOmega)}<r.
\ecases
\eeq
\end{lem}

\begin{lem}\lab{lem:w_m}  Assume \eqref{B} and fix $\ga\in C(\lbar\gO,\bbS^n(\gth_0))$.  For each $m>M_\ga$, there exists $w_m\in \Lip(\lbar\varOmega)$ 
and $\eta>0$ such that
\begin{equation}
0<\min_{\lbar\varOmega}w_m\leq \max_{\lbar\varOmega}w_m<m,
\end{equation}
and, in the viscosity supersolution sense,  
\begin{equation}\label{eq:lts0}
H_\ga(x,-Dw_m)\geq \eta \ \text{ in } \ \varOmega \ \text{ and } \  \
D^2w_m(x)\leq \eta^{-1}I \text{ in } \ \varOmega.
\end{equation}
\end{lem}

We continue with the proof  of  Proposition~\ref{thm:short} which parallels 
that of 
\cite{IS2015}*{Theorem 8}.

\bpr[Proof of Proposition~\ref{thm:short}] For $r\in(0,\,r_0)$ to be fixed below, 
let $v=v_r\in C^2(\bar \Omega)$ (for notational simplicity we omit the subscript $r$ in what follows) and $\eta>0$ be given 
by  Lemma~\ref{lem:v_r},  
set, for  $x\in\lbar\gO$,
\[   
w^\ep(x)
:=\exp\left(\fr{v(x)-m+2r}{\ep}\right), 
\]
compute, for any  $(x,t)\in Q$, 
\[\begin{aligned}
\ep\tr[a^\ep(x,t) &D^2 w^\ep]+b(x)\cdot Dw^\ep
\\&=\fr{w^\ep}{\ep}\left(a^\ep(x,t)Dv\cdot Dv+b\cdot Dv
+\ep \tr[a^\ep(x,t)D^2 v]\right)
\\&\leq \fr{w^\ep}{\ep}\left(\ga(x)Dv\cdot Dv+b(x)\cdot Dv
+\ep \tr[a^\ep(x,t)D^2 v]\right)
\\&\leq \fr{w^\ep}{\ep}\left(H_\ga(x,Dv)
+\ep \tr[a^\ep(x,t)D^2 v]\right).
\end{aligned}
\]
and  choose $\ep_0\in(0,\,1)$ so that 
\[
\ep_0\left(\tr a^\ep(x,t)D^2v\right)_+\leq \min\{\eta,\,r,\,1\};
\]
note that $\ep_0$ can be chosen so as to depend on $a^\ep$ only through $\gth_0$.   
\smallskip

 We assume henceforth that $\ep\in (0,\,\ep_0)$ and observe that, 
from the computation above, we get 
\beq\lab{eq:short-pde}
\ep\tr[a^\ep(x,t) D^2 w^\ep]+b(x)\cdot Dw^\ep \leq \bcases
0 \ \text{ for all }(x,t)\in \gO\setminus B_r\,\tim\,(0,\,\infty),\\[1mm]
\fr{2}{\ep}w^\ep \  \text{ for all }\ (x,t)\in B_r\,\tim\,(0,\infty).
\ecases
\eeq

 Let $C_0>0$ be a Lipschitz bound of $b$, and note 
that, if $H_\ga(x,p)\leq 0$, then $|p|\leq C_0\gth_0^{-1}$, which implies that 
$V_\ga(x)\leq C_0|x|^2/(2\gth_0)\leq C_0r^2/(2\gth_0)$ for all $x\in B_r$. 
We may thus assume by replacing, if needed, $r>0$ by a smaller number 
that $V_\ga\leq r$ in $B_r$. 
Accordingly we have 
\[
v-m+2r\leq V_\ga-m+3r\leq -m+4r \ \text{ in } \ B_r,
\]
and
\beq\lab{eq:short-ineq1}
w^\ep \leq \exp\left(\fr{-m+4r}{\ep}\right)  \ \text{ in } \ B_r.
\eeq

Observe also that 
\[
v-m+2r>V_\ga -m+r\geq r \  \text{ in } \  \lbar\gO\setminus \gS_\ga^m, 
\]
and
\beq\lab{eq:w^ep1}
w^\ep>\exp\left(\fr{r}{\ep}\right) \  \text{ in } \  \lbar\gO\setminus 
\gS_\ga^m.
\eeq

 Next set $\,d_\ep=\fr{2}\ep\exp(\fr{-m+4r}{\ep})\,$ and 
\[
z^\ep(x,t)=w^\ep(x)+d_\ep t \ \ \ \text{ for }\ 
(x,t)\in\lbar\gO\tim[0,\,\infty).
\]

It is immediate  from \eqref{eq:short-pde} and \eqref{eq:short-ineq1}
that  
\beq\lab{eq:short-pde1}
z^\ep_t\geq \ep\tr[a^\ep D^2z^\ep]+b\cdot Dz^\ep \ \text{ in } \ Q.
\eeq

 We choose $C_1>0$ so that, for all $\ep\in(0,\,1)$, 
\[
u^\ep \leq C_1 \  \text{ on  } \ \lbar Q,  
\] 
and by replacing,  if necessary, $\ep_0>0$ by a smaller number we may assume that, for all  $\ep\in(0,\,\ep_0)$, 
\[
C_1<\exp\left(\fr{r}{\ep}\right).
\]

 It follows  from \eqref{eq:w^ep1} that 
\[
z^\ep\geq w^\ep\geq \exp\left(\fr{r}{\ep}\right)>C_1
\geq u^\ep \ \text{ on} \ (\lbar\gO\setminus \gS_\ga^m)\tim[0,\,\infty);
\]
note that, since $m<M_\ga$, we have 
$\,\pl\gO\subset \lbar\gO\setminus \gS_\ga^m$. 
\smallskip

 On the other hand, for any $x\in \gS_\ga^m$, we have 
\[
z^\ep(x,0)=w^\ep(x)>0\geq u^\ep(x,0),
\]
and, hence,  
\[
u^\ep\leq z^\ep \ \text{ on } \ \pbr Q.
\]

  We find from the above,  \eqref{eq:short-pde1} and the comparison principle that 
\[
u^\ep\leq z^{\ep} \ \text{ on  }  \  \lbar Q,
\]
and, in particular, for any $t\in[0,\,\exp((m-5r)/\ep)]$, 
\[
u^\ep(0,t)\leq z^\ep(0,t)\leq w^\ep(0)+\fr 2\ep\exp\left(\fr {-r}\ep\right)
\leq
 \exp\left(\fr{-m+3r}{\ep}\right)+\fr 2\ep\exp\left(\fr {-r}\ep\right). 
\]

It is now clear that, for a given $\gd>0$, we may choose $r>0$ and $\ep_0\in(0,\,1)$ 
so that if $\ep\in(0,\,\ep_0)$, then 
\[
u^\ep(0,t)
\leq \exp\left(\fr{-m+3r}{\ep}\right)+\fr 2\ep\exp\left(\fr {-r}\ep\right)
<\gd\quad\text{ for all }\ t\in[0,\,\exp((m-\gd)/\ep)].\qedhere 
\]
\epr

We continue with

\bproof[Proof of Proposition~\ref{thm:long}]   
We fix $r\in(0,\,r_0)$ small enough so that, as in the previous proof, $V_\ga\leq r\,$ in $B_r,$ 
and $\,m-5r>M_\ga$.  
In view of 
Lemmata~\ref{lem:v_r} and \ref{lem:w_m}, 
we may choose $v\in C^2(\lbar\gO)$, $w\in\Lip(\lbar\gO)$ and $\eta>0$ so that, in addition to \eqref{takis10}, $0<\min_{\lbar\gO}w<\max_{\lbar\gO}w<m-5r$, 
and, in the viscosity supersolution sense,
\[
H_\ga(x,-Dw)\geq \eta \quad\text{ and }\quad 
D^2w\leq \eta^{-1}I \quad\text{ in }\ \gO. 
\] 

Setting $u=-w$, $\rho^-=\min_{\lbar\gO}w$ and $\rho^+=\max_{\lbar\gO}w$, 
we get that $\rho^+<m-5r$, \ $0>-\rho^-\geq u\geq -\rho^+$ on $\lbar\gO$ and, 
in the viscosity subsolution sense, 
\[
H_\ga(x,Du)\geq \eta \quad\text{ and }\quad 
D^2u\geq -\eta^{-1}I \quad\text{ in }\ \gO. 
\]

For $\ep\in(0,\,1)$, we set 
\[
z^\ep=-\exp\left(\fr{v-m+2r}{\ep}\right)+\exp\left(\fr{u}{\ep}\right)
-\exp\left(\fr{-\rho^-}{\ep}\right),
\]
and find that, in the viscosity subsolution sense,
\[\begin{aligned}
\ep\tr[a^\ep D^2z^\ep]+b\cdot Dz^\ep
\geq&  -\fr 1\ep \exp\left(\fr{v-m+2r}{\ep}\right)\left(
H_\ga(x,Dv)+\ep\tr[a^\ep D^2v]\right)
\\&+\fr 1\ep \exp\left(\fr{u}{\ep}\right)\left(H_\ga(x,Du)+\ep\tr[a^\ep D^2u]\right)
\quad\text{ in }\ Q.
\end{aligned}
\]

Let $\ep_0\in(0,\,1)$ be a constant to be specified later and assume henceforth that $\ep\in(0,\,\ep_0)$.  
Observing that   
in the viscosity subsolution sense, 
\[
\ep\tr[a^\ep D^2u]\geq -\eta^{-1}\tr a^\ep\geq -n(\gth_0\eta)^{-1}\quad\text{ in }\ Q,
\]
and 
\[
\tr[a^\ep D^2v]\leq \|D^2v\|_{L^\infty(\gO)}\tr a^\ep
\leq n\gth_0^{-1}\|D^2v\|_{L^\infty(\gO)} \quad\text{ in }\ Q,
\]
and, if for $x \in \lbar \varOmega$,  
\[\begin{aligned}
f(x):=&-\fr 1\ep \exp\left(\fr{v(x)-m+2r}{\ep}\right)\left(
H_\ga(x,Dv(x))+\ep n\gth_0^{-1}\|D^2v\|_{L^\infty(\gO)}\right)
\\&+\fr 1\ep \exp\left(\fr{u(x)}{\ep}\right)\left(\eta-\ep n(\eta\gth_0)^{-1}\right),
\end{aligned}
\]
we obtain, in the viscosity subsolution sense,
\beq\lab{eq:long1}
\ep\tr[a^\ep D^2z^\ep]+b\cdot Dz^\ep
\geq f \ \text{ in }\ Q.
\eeq

 Choosing $\ep_0\in(0,\,1)$ so that
\[
\ep_0n\gth_0^{-1}\|D^2v\|_{L^\infty(\gO)}\leq \min\{\eta,\,1\}
\ \text{ and } \ 
\ep_0 n(\eta\gth_0)^{-1}\leq \fr \eta 2,
\]
we get 
\[
\eta-\ep n(\eta\gth_0)^{-1}\geq \fr \eta 2 \ \ \ \text{and} \  \ \ 
H_\ga(x,Dv)+\ep n\gth_0^{-1}\|D^2v\|_{L^\infty(\gO)}
\leq \bcases
0&\text{ for all }\ x\in \gO\setminus B_r,\\
2&\text{ for all }\ x\in B_r,
\ecases
\] 
and, accordingly,
\begin{align*}
f
&\geq 
\bcases
0 &\text{ in } \ \gO\setminus B_r,\\[1mm]
\disp 
-\fr 2\ep \exp\left(\fr{-m+4r}{\ep}\right)+\fr{\eta}{2\ep}\exp\left(\fr{-\rho^+}{\ep}\right) 
&\text{ in  } \ B_r.
\ecases
\end{align*}

Since $\rho^+<m-5r$, we have 
\[\begin{aligned}
-2\exp\left(\fr{-m+4r}{\ep}\right)+\fr{\eta}{2}\exp\left(\fr{-\rho^+}{\ep}\right)
&\,\geq 
-2\exp\left(\fr{-\rho^+-r}{\ep}\right)+\fr{\eta}{2}\exp\left(\fr{-\rho^+}{\ep}\right)
\\&\,=
\exp\left(\fr{-\rho^+}{\ep}\right)\left(-2\exp\left(\fr{-r}{\ep}\right)+\fr{\eta}{2}\right).
\end{aligned}
\]

We may assume by replacing $\ep_0\in(0,\,1)$ by a smaller number 
that
\[
2\exp\left(\fr{-r}{\ep_0}\right)\leq \fr{\eta}{2},
\]
and, therefore,
\[
-2\exp\left(\fr{-m+4r}{\ep}\right)+\fr{\eta}{2}\exp\left(\fr{-\rho^+}{\ep}\right)\geq 0,
\]
which ensures that $f\geq 0$ in  $\gO$, and, hence,   $z^\ep$, as a function of $(x,t)\in Q$, is a subsolution of \eqref{eq:l-pde}.

 Next observe that 
\[
z^\ep<0 \  \text{ on } \  \lbar\gO, 
\]
and, if $V_\ga(x)>m$,   
\[
z^\ep(x)\leq -\exp\left(\fr{V_\ga(x)-m+r}{\ep}\right)
\leq -\exp\left(\fr{r}{\ep}\right).
\] 

Fix $C_1>0$ so 
that, for $\ep\in(0,\,1)$,  $u^\ep\geq -C_1$ on $\lbar Q$,
and,  assume henceforth 
that $\ep_0\in(0,\,1)$ is small enough so that
\[
\exp\left(\fr{r}{\ep_0}\right)\geq C_1. 
\]

Consequently, we have 
\[
z^\ep\leq
\bcases
-\exp\left(\fr{r}{\ep}\right)\leq -C_1\leq u^\ep 
&\text{ in } \ (\lbar \gO \setminus \gS_\ga^m)\tim[0,\,T(\ep)),\\
0\leq u^\ep(\cdot,0) &\text{ in }  \gS_\ga^m,\\
0\leq u^\ep &\text{ in} \ (\gS_\ga^m\cap\pl\gO)\tim(0,\,T(\ep)),
\ecases
\]
that is  
\[z^\ep\leq u^\ep \ \text{ on  } \  \pbr Q_{T(\ep)},
\]
and, hence, 
by the comparison principle, 
\[
z^\ep\leq u^\ep\ \text{ on } \  \lbar Q_{T(\ep)}. 
\]

Finally, we note that
\[\begin{aligned}
z^\ep(0)&\,=-\exp\left(\fr{v(0)-m+2r}{\ep}\right)
-\exp\left(\fr{-\rho^-}{\ep}\right)
\\&\,\geq -\exp\left(\fr{-m+4r}{\ep}\right)
-\exp\left(\fr{-\rho^-}{\ep}\right)
\to 0 \quad\text{ as }\ep\to 0,
\end{aligned}
\]
which completes the proof.  
\epr

The following corollary is a variant of \cite{IS2015}*{Theorem 4.1}. 
Since its proof is  similar to the one of Proposition~\ref{thm:long} above, here we present only an outline.

\begin{cor}\lab{cor:long} Assume \eqref{B}, fix $\ga\in C(\lbar\gO,\bbS^n(\gth_0))$,  $T(\ep)\in (0,\,\infty]$, $C>0$ and $m>M_\ga$. If, for  $a^\ep \in C(\bar Q_{T(\ep)};\bbS^n(\gth_0))$   
such that $a^\ep\geq \ga$ in $Q_{T(\ep)}$,  
$u^\ep\in C(\lbar Q_{T(\ep)})\cap C^{2,1}(Q_{T(\ep)})$  
is a supersolution of \eqref{eq:l-pde} in $Q_{T(\ep)}$ such that 
\[
u^\ep\geq 0\quad\text{ in } \ (\gS_\ga^m\cap\pl\gO)\tim(0,T(\ep)) \  \text{ and} \  u^\ep\geq -C \ \ \text{ in } \ Q_{T(\ep)}, 
\]
then, for any $\gd>0$, there exists $\ep_0\in(0,\,1)$ such that, if 
$\ep\in(0,\,\ep_0)$, then
\[
u^\ep(0,t)\geq -\gd \quad\text{ for all }\ t\in[\exp(m/\ep),\,T(\ep)).
\]
\end{cor}

\bpr[Outline of proof] Let $r\in(0,\,r_0)$, $\eta>0$, $\rho^\pm$, $v\in C^2(\lbar\gO)$, and $w,\,u,\,z^\ep,\,f\in \Lip(\lbar\gO)$ be the same as in the proof of Proposition
\ref{thm:long}. According to \eqref{eq:long1}, we have in the viscosity subsolution sense,
\[
\ep\tr[a^\ep D^2z^\ep]+b\cdot Dz^\ep\geq f \ \ \text{ in }Q.
\]

Choosing $\ep_0\in(0,\,1)$ so that 
\[
\ep_0n\gth_0^{-1}\|D^2v\|_{L^\infty(\gO)}
\leq \min\{\eta,\,1\}, \ \ 
\ep_0n(\eta\gth)^{-1}\leq \fr \eta 2 
\ \ \text{ and } \ \
2\exp\left(\fr{-r}{\ep_0}\right)\leq \fr{\eta}{4}
\]
and noting that, in $\gO\setminus B_r$ and $\ep\in(0,\,\ep_0)$,
\[
f\geq \fr 1\ep \exp\left(\fr{u}\ep\right)
\left(\eta-\ep n(\eta\gth)^{-1}\right)\geq 
\fr{\eta}2\exp\left(\fr{-\rho^+}\ep\right),
\]
we compute, as in the proof of Proposition~\ref{thm:long}, to get,  for any $\ep\in(0,\,\ep_0)$, that  
\[
f\geq \exp\left(\fr{-\rho^+}{\ep}\right)\left(-2\exp\left(\fr{-r}{\ep}\right)+\fr\eta 2\right) 
\geq \fr{\eta}{4}\exp\left(-\fr{\rho^+}{\ep}\right) \ \text{in } \ \gO. 
\]

Now, we fix $\gamma\in(0,\,\eta]$, set,   for
$(x,t)\in\lbar Q\ \text{ and }\ \ep\in(0,\,\ep_0),$
\[
g^\ep(x,t):=z^\ep(x)-C+\fr {\gamma t} 4\exp\left(-\fr{\rho^+}{\ep}\right),
\]
and observe that, for each $\ep\in(0,\,\ep_0)$, $g^\ep$ is
a subsolution of \eqref{eq:l-pde}.

Let
\[\tau(\ep)=\fr{4C}{\gamma}\exp(\fr{\rho^+}{\ep}),
\] 
and 
observe that, for any $(x,t)\in\lbar\gO\tim[0,\,\tau(\ep)]$ \ such that \ $V_\ga(x)>m$,   
\[
g^\ep(x,t)\leq z^\ep(x)\leq -\exp\left(\fr r\ep\right).
\] 

We may assume by replacing $\ep_0>0$ by a smaller number if necessary that 
\[
\exp\left(\fr r{\ep_0}\right)\geq C.
\]

Accordingly, we have
\[
g^\ep\leq
\bcases
-\exp\left(\fr{r}{\ep}\right)\leq -C\leq u^\ep
&\text{ in } \  (\lbar \gO \setminus \gS_\ga^m)\tim[0,\,T(\ep)\wedge\tau(\ep)),\\[2pt]
-C\leq u^\ep(\cdot,0) &\text{ in } \ \gS_\ga^m\tim\{0\},\\[2pt]
0\leq u^\ep &\text{ in } \ (\gS_\ga^m\cap\pl\gO)\tim(0,\,T(\ep)\wedge\tau(\ep)).
\ecases
\]

Thus,  
\[g^\ep\leq u^\ep \ \text{ in }\,  
\pbr Q_{T(\ep)\wedge\tau(\ep)},
\]
and, hence, by the comparison principle
\[
g^\ep\leq u^\ep \ \text{ in } \ \lbar 
Q_{T(\ep)\wedge\tau(\ep)}. 
\]

The final step begins by noting that
\[
u^\ep(0,\tau(\ep))\geq g^\ep(0,\tau(\ep))=z^\ep(0) 
\]
and 
\[
z^\ep(0)\geq -\exp\left(\fr{-m+4r}{\ep}\right)
-\exp\left(\fr{-\rho^-}{\ep}\right)
\to 0 \quad\text{ as }\ep\to 0. 
\]

Fix $\gd>0$ and, if necessary, replace $\ep_0$ 
by a smaller number such that 
$z^\ep(0)\geq -\gd$ for all $\ep \in(0,\,\ep_0)$.  
Recalling the definition of $\tau(\ep)$ and observing that 
\[
u^\ep(0,t)\geq z^\ep(0) \quad\text{ if }\ 
t=\fr{4C}\gamma 
\exp\left(\fr{\rho^+}{\ep}\right)<T(\ep)\ \text{ and }\ 0<\gamma\leq\eta, 
\]
we conclude that
\[
u^\ep(0,t)\geq -\gd \ \ \text{ for all }\ 
t\in[(4C/\eta)\exp(\rho^+/\ep),\,T(\ep)) 
\ \text{ and } \ \ep\in(0,\,\ep_0). 
\]

Since $\rho^+<m-5r$, by selecting $\ep_0\in(0,\,1)$ 
sufficiently small, we may assume that 
$(4C/\eta)\exp(\rho^+/\ep)\leq \exp(m/\ep)$ for all $\ep\in(0,\,\ep_0)$, 
which completes the proof. 
\epr

\section{The proofs of Propositions~\ref{thm:A}, \ref{thm:B} and \ref{thm:C}}

\bpr[Proof of Proposition~\ref{thm:A}] Since the arguments are similar 
for both cases when $\gb_1<\gb_2$ and $\gb_1>\gb_2$, here we treat only the case
$\gb_1<\gb_2$.

We argue by contradiction and suppose that
\beq\lab{eq:by-contra}
\limsup_{k\to \infty}\gl_k < M(\gb_2).
\eeq

 Let $\gd>0$ be a constant to be fixed later, define 
$\ga_\gd^+$ and  $\cH_\gd^+$ as in  
Section \ref{sec:quasi-p}, with $\gb_0$ replaced by $\gb_2$,  and, as in Section 
\ref{sec:quasi-p}, 
let $V_\gd^+$ be the maximal subsolution of 
\[
\cH_\gd^+(x,Du)=0 \ \  \ \text{ in } \ \gO, \ \ \ u(0)=0, 
\] 
and set $M_\gd^+=\min_{\pl\gO}V_\gd^+$. 
\smallskip

 Since Proposition~\ref{thm:quasi2-1}  
yields 
\[
\lim_{\gd\to 0+} M_\gd^+ =M(\gb_2), 
\]
in view of \eqref{eq:by-contra}, we may choose  $\gd>0$ so that 
\[
\limsup_{k\to\infty} \gl_k +\gd <M_\gd^+.  
\]

We fix $m\in\R$ so that 
\[\limsup_{k\to\infty} \gl_k +\gd<m 
<M_\gd^+,\] 
and, by passing to a subsequence if necessary,  we may assume that 
\[
\gl_k \leq m-\gd \quad\text{ for all }\ k\in\N.  
\]

Set 
\[
\gS=\{x\in\lbar\gO\mid V_\gd^+(x)\leq m\},
\]
and note that $\gS$ is a compact subset of $\gO$. 
\smallskip
 
 In view of the continuity of the map $t\mapsto u^\ep(0,t)$, 
reselecting,  if needed, $\gb_1$, $\mu_k$ and $\gl_k$, we may assume that, 
for  all $\,t\in[\exp(\mu_k/\ep_k),\,\exp(\gl_k/\ep_k)]$ and  $k\in\N$, 
\beq\lab{eq:(A)1}
\gb_2-\fr \gd 2 <\gb_1\leq u^{\ep_k}(0,t)\leq \gb_2 . 
\eeq

 Now we choose $\gamma\in(0,\,\gd/2)$ small enough, so that 
\beq\lab{eq:(A)4}
\gS\subset \gO_{\gamma}
\quad\text{ and }\quad 
\gb_2-\gb_1>2\gamma.
\eeq

 Proposition~\ref{thm:conclusion3} gives  $\ep_0\in(0,\,1)$ such that, 
if $\ep\in(0,\,\ep_0)$,
\beq\lab{eq:(A)2}
|u^\ep(x,t)-u^\ep(0,t)|<\gamma\quad\text{ for all }\ (x,t)\in\gO_\gamma\tim[\exp\left(a_1/\ep\right),\,\infty). 
\eeq

We assume that $\ep_k< \ep_0$ for all $k\in\N$, and 
combine \eqref{eq:(A)2} and \eqref{eq:(A)1}, to find  
\beq\lab{eq:(A)3}
|u^{\ep_k}(x,t)-\gb_2|\leq\gd \quad \text{ for all }\ (x,t)\in\gO_\gamma\tim
[\exp(\mu_k/\ep_k),\,\exp(\gl_k/\ep_k)],
\eeq
and 
\[
u^{\ep_k}(x,\exp(\mu_k/\ep_k))\leq \gb_1+\gamma 
\quad\text{ for all }\ x\in\gO_\gamma \ \text{ and }\ k\in\N.
\]

Since  \eqref{eq:(A)3} implies that 
\[
a(x,u^{\ep_k}(x,t))\leq \ga_\gd(x) \ \  \text{ for all }\ 
(x,t)\in \gO\tim[\exp(\mu_k/\ep_k),\,
\exp(\gl_k/\ep_k)],\ k\in\N,
\]
setting 
\[
\bcases
v^k(x,t)=u^{\ep_k}(x,t+\exp(\mu_k/\ep_k))-\gb_1-\gamma, &
\\[1mm]
 a^k(x,t)=a(x,u^{\ep_k}(x,t+\exp(\mu_k/\ep_k))),&
\ecases\]
we see that 
\[
v^k_t=\ep_k\tr[a^k(x,t)D^2v^k]+b(x)\cdot Dv^k
\quad\text{ for all }\ (x,t)\in Q.
\]

Furthermore, since $v^k(\cdot,0)\leq 0$ in $\gO_\gamma,$
it follows that 
\[
v^k(\cdot ,0)\leq 0 \ \text{ in  } \ \gS.
\]

An application of Proposition~\ref{thm:short}, with $\ep_k$, $v^k$ and $\gamma$  
in place of $\ep$, $u^\ep$ and $\gd$ respectively, guarantees that,
for sufficiently large $k$, we have
\[
v^k(0,t)\leq \gamma \quad \text{ for all }\ t
\in [0,\,\exp(\gl_k/\ep_k)-\exp(\mu_k/\ep_k)],
\] 
which, in particular, yields
\[
v^k(0,\exp(\gl_k/\ep_k)-\exp(\mu_k/\ep_k))\leq \gamma.
\]

This shows that
\[
u^{\ep_k}(0,\exp(\gl_k/\ep_k))\leq \gb_1+2\gamma<\gb_2,
\] 
which is a contradiction. 
\epr

\bpr[Proof of Proposition~\ref{thm:B}] Since the arguments are similar, here we only consider the case 
where $\gb_2<\gb_1$ holds.
\smallskip

 We suppose that 
\beq\lab{eq:(B)1}
G^-(\gb_2)>\gb_2,
\eeq
and obtain a contradiction. 

  For a small constant $\gd>0$ to be chosen later,  define  $\ga_\gd^-$ and  $\cH_\gd^-$  as in 
Section \ref{sec:quasi-p}, with $\gb_0$ replaced by $\gb_2$,  let $V_\gd^-$ be  the quasi-potential corresponding to $(\ga_\gd^-,b)$, that is 
the maximal subsolution of 
\[
\cH_\gd^-(x,Du)=0 \ \  \ \text{ in } \ \gO\quad\text{ and } \ \ \ u(0)=0. 
\]  
and  $V^{\gb_2}$ the quasi-potential corresponding to 
the pair $(a(\cdot,\gb_2),b)$,  
set 
\[
M_\gd^-=\min_{\pl\gO}V_\gd^-,\quad 
\gG_\gd^-=\argmin(V_\gd^-|\pl\gO) \quad \text{ and }\quad 
\gG^{\gb_2}=\argmin(V^{\gb_2}|\pl\gO),
\]
and observe that
\[
G^-(\gb_2)=\min_{\gG^{\gb_2}}g. 
\] 

Due to \eqref{eq:(B)1}, we have
\[
\min_{\gG^{\gb_2}}g>\gb_2. 
\]

Furthermore, in view of \eqref{eq:quasi2-5}, we may choose $\gd>0$ so that 
\beq\lab{eq:(B)0}
\min_{\gG_\gd^-}g>\gb_2+\gd.
\eeq

 Finally replacing, if necessary, $\gb_1$, $\mu_k$ and $\gl_k$  we may assume  that
\[
\gb_1\geq u^\ep(0,t)\geq \gb_2 \ \ \text{ for all }\ 
t\in[\exp(\mu_k/\ep_k),\,
\exp(\gl_k/\ep_k)],\ k\in\N, 
\]
and 
\beq\lab{eq:(B)-1}
\gb_1<\gb_2+\gd/2. 
\eeq

 Since, by the maximum principle, $g_{\min}\leq u^\ep\leq g_{\max}$ 
in $\lbar Q$, we find that  
Theorem~\ref{thm:conclusion3} yields  $\ep_0\in(0,\,1)$ such that, if $\ep\in(0,\,\ep_0)$, then  
\beq\lab{eq:(B)1-2}
|u^\ep-u^\ep(0,t)|<\gd/2 \  \text{ in }\ 
 \gO_{\gd/2}\tim[\exp(a_1/\ep),\,\infty). 
\eeq

Consequently, if $k\in\N$ is sufficiently large, then 
$\,\ep_k<\ep_0\,$ and 
\beq\lab{eq:(B)2}
|u^{\ep_k}-\gb_2|<\gd \ \text{ in} \ 
\gO_{\gd/2}\tim[\exp(\mu_k/\ep_k),\,\exp(\gl_k/\ep_k)].
\eeq

Henceforth,  passing if necessary to a subsequence, we assume that \eqref{eq:(B)2} 
holds for all $k\in\N$ and, thus 
\beq\lab{eq:(B)3}
\ga_\gd^-(x)\leq a(x,u^{\ep_k}(x,t))\quad\text{ for all }\ 
(x,t)\in\lbar\gO\tim[\exp(\mu_k/\ep_k),\,\exp(\gl_k/\ep_k)],\ k\in\N.
\eeq

We set $\,\Pi=\{x\in\pl\gO\mid g(x)> \gb_2+\gd\}$ 
and note that, in view of \eqref{eq:(B)0}, $\Pi$ is an open neighborhood, relative to $\pl\gO$,
of $\gG_\gd^-$ and 
\[
\{x\in\lbar\gO\mid V_\gd^-(x)\leq M_\gd^-\}
=\{x\in\gO\mid V_\gd^-(x)\leq M_\gd^-\}\cup\gG_\gd^-
\subset \gO^\Pi, 
\] 
and deduce, for $\gamma>0$ sufficiently small, 
\beq\lab{May9}
\{x\in\lbar\gO\mid V_\gd^-(x)\leq M_\gd^-+\gamma\}
\subset \gO^\Pi_{\gamma}. 
\eeq

In view of \eqref{eq:(B)-1},  we observe that 
\beq\lab{eq:(B)-2}
g>\gb_2+\gd>\gb_1 \ \text{ in }\  \Pi.
\eeq

We fix $\gamma>0$ so that \eqref{May9} and $\,5\gamma<\gb_1-\gb_2\,$ hold, set  
\[
\gS:=\{x\in \lbar\gO\mid V_\gd^-(x)\leq M_\gd^-+\gamma\} \subset\gO^\Pi_\gamma.
\] 

Noting that $\gb_1>\gb_1-4\gamma>\gb_2$, we 
select a sequence $\{\nu_k\}$ so that 
\beq\lab{eq:(B)4}
\bcases
\mu_k<\nu_k<\gl_k,\qquad 
u^{\ep_k}(0,\exp(\nu_k/\ep_k))=\gb_1-3\gamma\quad\text{ for all }\ k\in\N,&\\[1mm]
\gb_1\geq u^{\ep_k}(0,t)\geq \gb_1-3\gamma \quad \text{ for all }\ 
t\in[\exp(\mu_k/\ep_k),\,\exp(\nu_k/\ep_k)],\ k\in\N. 
\ecases
\eeq

Now we show that, for some $\rho>0$ and sufficiently large $k\in\N$,
\beq\label{eq:(B)5}
\exp(\nu_k/\ep_k)\geq \exp(\mu_k/\ep_k)+\exp(\rho/\ep_k). 
\eeq

For this, similarly  to \eqref{eq:(B)1-2}, we use Proposition~\ref{thm:conclusion3}, 
to find that,  for some $r\in(0,\,r_0)$ and sufficiently large $k\in\N$, 
\[
|u^{\ep_k}-u^{\ep_k}(0,\cdot)|<  \gamma \ \text{ in }\ 
B_r\tim[\exp(a_1/\ep_k), \infty).
\] 

For every such large $k\in\N$, we set
\[
v^k(x,t):=u^{\ep_k}(x,t+\exp(\mu_k/\ep_k))-\gb_1+\gamma \quad\text{ for } \ 
(x,t)\in \lbar Q,
\]
and note that 
\[
v^k(\cdot,0)\geq 0 \ \text{ in }\ B_r.
\]

We apply Proposition~\ref{thm:short}, with $\ep$, $u^\ep$ and $\ga$ 
replaced respectively by $\ep_k$, $-v^k$ and $\gth_0^{-1}I$, to deduce that,
for some $\rho>0$,
\[
-v^k(0,t)\leq \gamma \quad\text{ for all }\ t\in[0,\,\exp(\rho/\ep_k)],
\]    
that is, 
\[
u^{\ep_k}(0,t)\geq \gb_1-2\gamma\quad\text{ for all }\ t\in 
[\exp(\mu_k/\ep_k),\,\exp(\mu_k/\ep_k)+\exp(\rho/\ep_k)], 
\]
which, in view  of the choice of $\nu_k$,  implies that 
\eqref{eq:(B)5} holds for sufficiently large $k\in\N$. 

In what follows we may assume by replacing if necessary $\{\ep_k\}$ by a further subsequence 
that \eqref{eq:(B)5} is satisfied for some $\rho>0$ and all $k\in\N$. 
We set 
\[
w^k(x,t)=u^{\ep_k}(x,t+\exp(\mu_k/\ep_k))-\gb_1+3\gamma 
\quad\text{ for }\ (x,t)\in\lbar Q,\ k\in\N,
\]
and note that, in view of  \eqref{eq:(B)4} and \eqref{eq:(B)-2}, 
\[
\bcases
w^k(0,t)\geq 0 & \text{ for all }\ t\in[0,\, \exp(\nu_k/\ep_k)-\exp(\mu_k/\ep_k)]
\\[3pt]
w^k(x,t)=g(x)-\gb_1+3\gamma\geq 0 & \text{ for all }\ 
(x,t)\in\Pi\tim[0,\,\infty).
\ecases
\]

Recalling \eqref{eq:(B)5}, we apply 
Theorem~\ref{thm:conclusion}, with $\ep$ and $u^\ep$ 
replaced by $\ep_k$ and $-w_k$, to get,  for sufficiently large $k$,
\[
-w^k(x,\exp(\nu_k/\ep_k)-\exp(\mu_k/\ep_k))\leq \gamma 
\quad\text{ for all }\ x\in \gO^\Pi_\gamma,
\] 
which reads 
\[
u^{\ep_k}(x,\exp(\nu_k/\ep_k))\geq \gb_1-4 \gamma
\quad\text{ for all }\ x\in \gO^\Pi_\gamma. 
\]

Finally, for $(x,t)\in\lbar Q,$ we set 
\[
z^k(x,t)=u^{\ep_k}(x,t+\exp(\nu_k/\ep_k))-\gb_1+4 \gamma, 
\]
observe that, if $k\in\N$ is sufficiently large, then
\[ z^k(\cdot ,0)\geq 0 \ \text{ in }\  \gS \ \text{ and} \ 
z^k=g-\gb_1+4 \gamma\geq 0 \ \text{ in } \ 
\Pi\tim[0,\,\infty),
  \]
and invoke Proposition~\ref{thm:long}, to conclude 
that, for sufficiently large $k\in\N$, 
\[
z^k(0,\exp(\gl_k/\ep_k)-\exp(\nu_k/\ep_k)) \geq -\gamma,
\]
and, hence,
\[
u^{\ep_k}(0,\exp(\gl_k/\ep_k))\geq \gb_1-5\gamma>\gb_2,
\]
which  is a contradiction. 
\epr

\bpr[Proof of Proposition~\ref{thm:C}] Since the arguments are similar, we give the proof under the assumption that 
\beq\label{eq:C3}
G^-(\gb_0)>\gb_0.
\eeq

We suppose that 
\beq\lab{eq:C4} 
\rho_0>M(\gb_0),
\eeq
and obtain a contradiction. 

Fix $\gd>0$ 
and let $\ga_\gd^-$ and  $\cH_\gd^-$  as in 
Section \ref{sec:quasi-p} and $V_\gd^-$ and $V^{\gb_0}$ be the quasi-potentials corresponding to $(\ga_\gd^-,b)$ and $(a(\cdot,\gb_0),b)$ respectively, 
set
\[
M_\gd^-=\min_{\pl\gO}V_\gd^-,\quad 
\gG_\gd^-=\argmin(V_\gd^-|\pl\gO) \quad \text{ and }\quad 
\gG^{\gb_0}=\argmin(V^{\gb_0}|\pl\gO),
\]
and note, in view  of Proposition~\ref{thm:quasi2-1}, \eqref{eq:C3} and \eqref{eq:C4}, that
\[
\liminf_{\gd\to 0+}\min_{\gG_\gd^-}g\geq \min_{\gG^{\gb_0}}g 
=G^-(\gb_0)>\gb_0 \ \text{and } \ \lim_{\gd\to 0+}M_\gd^-=M(\gb_0)<\rho_0.
\] 
Choose $\gd>0$ so that
\beq\label{eq:C5}
\min_{\gG_\gd^-}g>4\gd+\gb_0.
\eeq

For $m> M_\gd^-$ set 
\[
\gS^m:=\{x\in\lbar\gO\mid V_\gd^-(x)\leq m\}
\]
and note that
\[
\limsup_{m\to M_\gd^-+0}\gS^m\cap\pl\gO
= \{x\in\lbar\gO\mid V_\gd^-(x)\leq M_\gd^-\}\cap\pl\gO=\gG_\gd^-.
\]

Hence, we may choose $m\in(M_\gd^-,\,\rho_0]$ so that  
\beq\label{eq:C6}
\gS^m\cap\pl\gO\subset \{x\in\pl\gO\mid g(x)>\gb_0+3\gd\}. 
\eeq

The maximum principle yields  that,   for all 
$(x,t)\in Q$ and $\ep\in(0,\,1)$, $u^\ep(x,t)\in I_g$., while Theorem \ref{thm:conclusion3} implies the existence of  
$\ep_0\in(0,\,1)$ such that for all 
$(x,t)\in \gO_{\gd/2}\tim[\exp((\rho_0-\gd)/\ep),\,\infty)$ 
and $\ep\in(0,\,\ep_0)$, 
\[
|u^\ep(x,t)-u^\ep(0,t)|<\fr \gd 2.  
\]

Our assumptions yield $\gamma>0$ and a sequence $\{\ep_k\}\subset (0,\ep_0)$ such that $\lim_{k \to \infty} \ep_k=0$ and, for all $\rho\in[\rho_0-\gamma,\,\rho_0+\gamma]
\subset (0,\,\infty)$ and $k\in\N$, 
\[
u^{\ep_k}(0,\exp(\rho/\ep_k))\in [\gb_0-\gd/2,\,\gb_0+\gd/2].
\] 

Hence, if $(x,\rho)\in \gO_{\gd/2}\tim[\rho_0-\gamma,\,\rho_0+\gamma]$ 
and $k\in\N$, we get 
\beq\label{eq:C7}
u^{\ep_k}(x,\exp(\rho/\ep_k))\in(\gb_0-\gd,\,\gb_0+\gd),
\eeq
and, moreover,
\beq\label{eq:C8} 
\ga_\gd^-(x)\leq a(x,u^{\ep_k}(x,\exp(\rho/\ep_k)).  
\eeq

Set 
\[
v^\ep(x,t):=u^{\ep}(x,t+\exp((\rho_0-\gamma)/\ep))
-\gb_0-3\gd
 \ \ \text{ for all }(x,t)\in Q,\  
\ep\in(0,\,1),
\]
and
\[
a^\ep(x,t):=a(x,u^{\ep}(x,t+\exp((\rho_0-\gamma)/\ep)))\ \ 
\text{ for all }(x,t)\in Q,\  
\ep\in(0,\,1),
\]
and observe that, 
for  $T_k:=\exp((\rho_0+\gamma)/\ep_k)-\exp((\rho_0-\gamma)/\ep_k)$, $v^\ep$ is a solution of \eqref{eq:l-pde} and \eqref{eq:ibv-g}, 
\[
\ga_\gd^-(x)\leq a^{\ep_k}(x,t) \ \ \ \text{ for all } 
(x,t)\in\gO\tim [0,\,T_k],\ k\in\N, 
\]
and
\[
v^{\ep_k}(x,t)=g(x)-\gb_0-3\gd>0 \ \ \ \text{ for all } 
(x,t)\in\gS^{m}\cap \pl\gO \tim[0,\,T_k]
\]

In view of Corollary \ref{cor:long}, we may assume, by passing to a subsequence, that
\[
v^{\ep_k}(0,t)\geq -\gd  \ \ \ \text{ for all } 
t\in[\exp(m/\ep_k),\,T_k] \ \text{ and} \ k\in\N.
\]

Since $T_k> \exp(m/\ep_k)$ for sufficiently large $k\in\N$, 
we find $k\in\N$ such that
\[
v^{\ep_k}(0,T_k)\geq -\gd, 
\]
which yields the contradiction
\[
u^{\ep_k}(0,\exp((\rho_0+\gamma)/\ep_k))\geq \gb_0+3\gd.
\] 
\epr

\section{The proof of the main theorem} 

The proof of Theorem~\ref{thm:FK} is a relatively easy consequence of 
Propositions~\ref{thm:A}, \ref{thm:B} and \ref{thm:C} 
as shown in \cites{FK2010, FK2012a}.  
For the reader's convenience, we reproduce it here. 
We begin with two lemmata.

\begin{lem} \lab{lem:main} Assume 
\eqref{B} and let
 $u^\ep\in C(\lbar Q)\cap C^{2,1}(Q)$ be a solution of 
\eqref{eq:ql-pde} and \eqref{eq:ibv-g}. For any $\gd>0$ there exist 
$\gl_0>0$ and $\ep_0\in(0,\,1)$ such that 
\beq\lab{eq:main1}
|u^\ep(0,t)-g(0)|\leq \gd \quad\text{ for all }\ t\in[0,\,\exp(\gl_0/\ep)] 
\ \ \text{ and } \ \ \ep\in(0,\,\ep_0). 
\eeq
\end{lem}

\bpr Let $V\in\Lip(\lbar\gO)$ be the quasi-potential associated with 
$(\gth_0^{-1}I, b)$. We choose $m>0$ small enough 
so that $m<\min_{\pl\gO}V$ and  
\[
\{x\in\gO\mid V(x)\leq m\}\subset\{x\in\gO\mid |g(x)-g(0)|\leq \gd/2\}. 
\]

Applying Proposition~\ref{thm:short}, with $a^\ep(x,t)=a(x,u^\ep(x,t))$ and 
$\ga(x)=\gth_0^{-1} I$ and 
$u^\ep$ replaced by $\pm(u^\ep-g(0))-\gd/2$, we get  that, 
for each $\gamma>0$, there is $\ep_0\in(0,\,1)$ such that 
\[
\pm(u^\ep(0,t)-g(0))-\gd/2\leq \gamma 
\quad\text{ for all }\ t\in[0,\,\exp((m-\gamma)/\ep)] \ \ 
\text{ and } \ \ \ep\in(0,\,\ep_0).  
\] 

We fix $\gamma>0$ small enough so that 
$\gamma< \min\{\gd/2,\,m\}$, and we get   \eqref{eq:main1} with $\gl_0=m-\gamma$.  
\epr

\blem\label{lem:main'} Assume \eqref{B} and \eqref{G} and let $\gl>0$ and, for each $\ep\in(,\,1)$, $u^\ep\in C(\lbar Q)\cap C^{2,1}(Q)$ be a solution of \eqref{eq:ql-pde} and \eqref{eq:ibv-g}. 
If $c_1>c_0$, then 
\beq\label{eq:main'2}
\liminf_{\ep\to 0+}u^\ep(0,\exp(\gl/\ep))\geq \bar c(\gl),  
\eeq
and, if $c_1<c_0$, then 
\beq\label{eq:main'2+1}
\limsup_{\ep\to 0+}u^\ep(0,\exp(\gl/\ep))\leq \bar c(\gl).
\eeq
\elem 

This lemma is exactly the same as \cite{FK2012a}*{Lemma 3.12}.  

\bpr We give only the proof of the first assertion, since the other claim 
can be proved similarly. 

Note that, in view of the definition of $c_1$ and the function $\bar c$, 
$G^-(c)>c$ for all $c\in[c_0,\,c_1)$ and $\gl\not=M(c)$ for all 
$c\in[c_0,\,\bar c(\gl))$. Furthermore, 
since the function $M$ is continuous, we have $\gl>M(c)$ for all 
$c\in[c_0,\,\bar c(\gl))$.

We show first that, for any $\rho>0$, 
\beq\label{eq:main'3}
\liminf_{\ep\to 0+}u^\ep(0,\exp(\rho/\ep))\geq c_0.  
\eeq

According to Lemma \ref{lem:main}, there exists $\gl_0>0$ such that, for any
$\rho\in(0,\,\gl_0]$,  
\beq\label{eq:main'4}
\lim_{\ep\to 0+}u^\ep(0,\exp(\rho/\ep))= c_0,
\eeq 
which shows that  \eqref{eq:main'3} holds  if  $\rho\leq \gl_0$.

Fix any $\rho>\gl_0$ and, to prove \eqref{eq:main'3}, suppose to the contrary that 
\[
\liminf_{\ep\to 0+}u^\ep(0,\exp(\rho/\ep))< c_0.
\]

It is easily seen that there exist sequences 
$\{\ep_k\}$, $\{\mu_k\}$ and $\{\gl_k\}$ of positive numbers 
and two constants $\gb_1,\,\gb_2\in I_g$  
such that 
$\,\lim_{k\to\infty}\ep_k=0$, $\,c_0>\gb_1>\gb_2$,
and, for all $k\in\N$, 
\[
\gl_0<\mu_k<\gl_k\leq \rho, \quad 
u^{\ep_k}(0,\exp(\mu_k/\ep_k))=\gb_1 
\ \ \ \text{ and } \ \ \ u^{\ep_k}(0,\exp(\gl_k/\ep_k))=\gb_2. 
\]

Since $G^-(c_0)>c_0$ and $G^-$ is lower semicontinuous, 
we may assume  reselecting $\gb_1,\,\gb_2$ close enough to $c_0$ so   
that $G^-(\gb_2)>\gb_2$. This contradicts Proposition \ref{thm:B}, 
which proves that \eqref{eq:main'3} holds.    

To show \eqref{eq:main'2}, in view of \eqref{eq:main'3}, 
we may assume that $\bar c(\gl)>c_0$ and suppose that \eqref{eq:main'2} is false,
that is, 
\beq\label{eq:main'5}
\liminf_{\ep\to 0+}u^\ep(0,\exp(\gl/\ep))<\bar c(\gl),
\eeq
which in turn  implies together with \eqref{eq:main'3} that $\bar c(\gl)>c_0$. 

We set 
\[
\hat c(\rho):=\liminf_{\ep\to 0+}u^\ep(0,\exp(\rho/\ep)) \ \ \ \text{ for }
\rho\in (0,\,\gl],
\]
and show, arguing by contradiction,  that 
\beq\label{eq:main'6}
\hat c(\rho_1)\leq\hat c(\rho_2) \ \ \ \text{ if } \ 0<\rho_1<\rho_2\leq\gl. 
\eeq

To  this end, we suppose to the contrary that there exist $0<\rho_1<\rho_2\leq \gl$ 
such that  
\[
\hat c(\rho_1)>\hat c(\rho_2). 
\]

We may assume that $\hat c(\rho_2)\leq \hat c(\gl)$. Indeed, 
if $\hat c(\rho_2)>\hat c(\gl)$, then,  replacing $\rho_2$ by $\gl$,  
we find $\hat c(\rho_1)>\hat c(\rho_2)$ and $\hat c(\rho_2)=\hat c(\gl)$.  
Now, noting  by \eqref{eq:main'5} and \eqref{eq:main'4} that 
$c_0\leq \hat c(\rho_2)\leq \hat c(\gl)<\bar c(\gl)$, we may choose 
sequences   
$\{\ep_k\},\,\{\mu_k\},\,\{\gl_k\}$ of positive numbers and 
constants $\gb_1,\,\gb_2\in I_g$ such that 
$\lim_{k\to\infty}\ep_k=0$, $\bar c(\gl)>\gb_1>\gb_2>c_0$, 
and, for all $k\in\N$, 
\[
u^{\ep_k}(0,\exp(\mu_k/\ep_k))=\gb_1,\ \ 
u^{\ep_k}(0,\exp(\gl_k/\ep_k))=\gb_2 \ \ \text{ and } \ \ 
\rho_1<\mu_k<\gl_k\leq \rho_2.
\]

Here and there, for notational 
simplicity, we use the same symbols 
$\gb_i$, $\ep_k$, $\mu_k$ and $\gl_k$ to denote different quantities 
in different arguments. Moreover, 
since $c_0<\gb_2<\bar c(\gl)$, we have $G^-(\gb_2)>\gb_2$. 
Thus, we are in the situation that contradicts Proposition \ref{thm:B}, 
and we conclude that \eqref{eq:main'6} holds.

The last step of our proof is an application of Proposition \ref{thm:C}
for a contradiction. 

In view of the monotonicity \eqref{eq:main'6}, the function $\hat c$ 
has  at most countably many discontinuities on $(0,\,\gl]$ and,  recalling  that 
$c_0\leq \hat c(\rho)<\bar c(\gl)$ for all $\rho\in(0,\,\gl]$ and that 
$G^-(c)>c$ and $M(c)<\gl$ for all $c\in[c_0,\,\bar c(\gl))$,  
we may choose $\rho_0\in (0,\,\gl)$ so that $\hat c$ is continuous at $\rho_0$ 
and, for $\gb_0:=\hat c(\rho_0)$, 
\beq\label{eq:main'7}
\rho_0>M(\gb_0). 
\eeq

We fix any $\gd>0$. 
Since $G^-(\gb_0)>\gb_0$,  in view of the lower semicontinuity of $G^-$, 
we may choose $\gd_1\in(0,\,\gd/3)$ so that 
\beq\label{eq:main'8}
G^-(c)>\gb_0+3\gd_1 \ \ \ \text{ for all }c\in[\gb_0-3\gd_1,\,\gb_0+3\gd_1].
\eeq
 
Moreover  the continuity of $\hat c$ at $\rho_0$ yields 
$\gamma>0$ such that  
\[
[\rho_0-\gamma,\,\rho_0+\gamma]\subset (0,\,\gl),
\]
and, for all $\rho\in[\rho_0-\gamma,\,\rho_0+\gamma]$,
\beq\label{eq:main'9}
\hat c(\rho)\in[\gb_0-\gd_1,\,\gb_0+\gd_1]. 
\eeq

Now, we show that there exists a sequence $\{\ep_k\}\subset (0,\,1)$ 
such that $\lim_{k \to \infty}\ep_k=0$ and, for all $\rho\in[\rho_0-\gamma,\,\rho_0+\gamma]$ and $k\in\N$,
\beq\label{eq:main'10}
u^{\ep_k}(0,\exp(\rho/\ep_k))\in [\gb_0-3\gd_1,\,\gb_0+3\gd_1].
\eeq

Indeed, in view of the definition of $\hat c$ and \eqref{eq:main'9}, 
we may choose a sequence 
$\{\ep_k\}\subset (0,\,1)$ such that $\lim_{k \to \infty}\ep_k=0$ and, 
for all $k\in\N$,
\beq\label{eq:main'101}
u^{\ep_k}(0,\exp((\rho_0+\gamma)/\ep_k))\in [\gb_0-2\gd_1,\,\gb_0+2\gd_1].
\eeq

Since $\hat c(\rho_0-\gamma)\geq \gb_0-\gd_1$, \eqref{eq:main'9} gives 
\[
\liminf_{k\to\infty}u^{\ep_k}(0,\exp((\rho_0-\gamma)/\ep_k))\geq \gb_0-\gd_1,
\]
and, therefore, by passing to a subsequence if necessary, we may assume that, for all 
$k\in\N$, 
\beq\label{eq:main'102}
u^{\ep_k}(0,\exp((\rho_0-\gamma)/\ep_k))\geq \gb_0-2\gd_1.
\eeq

To complete the proof of \eqref{eq:main'10}, we need only to show
that for infinitely many $k\in\N$ and all $\rho\in[\rho_0-\gamma,\,\rho_0+\gamma]$, 
\beq\label{eq:main'11}
u^{\ep_k}(0,\exp(\rho/\ep_k))\in[\gb_0-3\gd_1,\,\gb_0+3\gd_1]. 
\eeq

If this is not the case, there exist a subsequence of $\{\ep_k\}$, 
which we denote again by the same symbol, and a sequence 
$\{\rho_k\}\subset [\rho_0-\gamma,\,\rho_0+\gamma]$ such that either 
\beq\label{eq:main'case1}
u^{\ep_k}(0,\exp(\rho_k/\ep_k))>\gb_0+3\gd_1 \ \ \ \text{ for all } k\in \N,
\eeq
or 
\beq \label{eq:main'case2}
u^{\ep_k}(0,\exp(\rho_k/\ep_k))<\gb_0-3\gd_1 \ \ \ \text{ for all } k\in \N,
\eeq

In view of  \eqref{eq:main'101}, 
if \eqref{eq:main'case1} holds, 
then there are two sequences $\{\mu_k\},\,\{\gl_k\}\subset (0,\,\gl)$ such that,
for all $k\in\N$, 
\beq\label{eq:main'12}
\bcases
\gb_0+3\gd_1=u^{\ep_k}(0,\exp(\mu_k/\ep_k))>
u^{\ep_k}(0,\exp(\gl_k/\ep_k))=\gb_0+2\gd_1,
\\ 
\rho_0-\gamma\leq \mu_k<\gl_k\leq \rho_0+\gamma.
\ecases\eeq

Similarly, in view of \eqref{eq:main'102}, if \eqref{eq:main'case2} holds, 
then there are two sequences $\{\mu_k\},\,\{\gl_k\}\subset (0,\,\gl)$ such that,
for all $k\in\N$, 
\beq\label{eq:main'13}
\bcases
\gb_0-2\gd_1=u^{\ep_k}(0,\exp(\mu_k/\ep_k))>
u^{\ep_k}(0,\exp(\gl_k/\ep_k))=\gb_0-3\gd_1,
\\
\rho_0-\gamma\leq \mu_k<\gl_k<\rho_0+\gamma.
\ecases
\eeq

If \eqref{eq:main'12} holds, 
setting $\gb_1:=\gb_0+3\gd_1$ and $\gb_2:=\gb_0+2\gd_1$ 
and noting by \eqref{eq:main'8} that $G^-(\gb_2)>\gb_2$, we apply Proposition 
\ref{thm:B}, to obtain a contradiction. 

In the case \eqref{eq:main'13} holds, 
setting $\gb_1:=\gb_0-2\gd_1$ and $\gb_2:=\gb_0-3\gd_1$ 
and noting that $G^-(\gb_2)>\gb_2$, we get a contradiction by 
Proposition \ref{thm:B}. 

Now, we find that \eqref{eq:main'11} holds 
and, therefore, there is a sequence $\{\ep_k\}\subset (0,\,1)$ 
for which \eqref{eq:main'10} holds. 

Thus, under the supposition \eqref{eq:main'5}, we have shown that 
\eqref{eq:main'7} and \eqref{eq:main'9} hold 
for some sequence $\{\ep_k\}\subset (0,\,1)$ converging to zero. 
Proposition \ref{thm:C} assures that $\rho_0\leq M(\gb_0)$, which contradicts 
\eqref{eq:main'7}. Therefore, we conclude that \eqref{eq:main'2} must hold. 
\epr

\bpr[Proof of Theorem~\ref{thm:FK}] In view of Theorem~\ref{thm:conclusion3}, 
we only need to show that
\beq\lab{eq:main2}
\lim_{\ep\to 0}u^\ep(0,\exp(\gl/\ep))=\bar c(\gl).
\eeq

\smallskip

 The  comparison principle yields that  
\[g_{\min}\leq u^\ep
\leq g_{\max}  \ \text {on } \ \lbar Q.
\]

 We fix $\gl>0$ and consider first the case $\gl<M(c_0)$, which implies  
that $\bar c(\gl)=c_0$, and  prove 
that 
\beq\lab{eq:main3}
\limsup_{\ep\to 0}u^\ep(0,\exp(\gl/\ep))\leq \bar c(\gl)=c_0.
\eeq 

We argue by contradiction and  suppose that 
\[
\limsup_{\ep\to 0}u^\ep(0,\exp(\gl/\ep))> c_0.
\]

 Using 
the continuity of the function $M$, we choose $\gb_1,\gb_2\in\R$ so that 
\beq\lab{eq:main4}
c_0<\gb_1<\gb_2<\limsup_{\ep\to 0}u^\ep(0,\exp(\gl/\ep)) 
\ \ \text{ and } \ \ M(\gb_2)>\gl,
\eeq
and note that, in view of Lemma~\ref{lem:main}, there are constants $\gl_0\in(0,\,\gl)$ 
and $\ep_0\in(0,\,1)$ such that 
\beq\lab{eq:main5}
u^\ep(0,\exp(\gl_0/\ep))\leq \gb_1\quad\text{ for all }\ep\in(0,\,\ep_0). 
\eeq

 On the other hand,  \eqref{eq:main4} yields a sequence $\{\ep_k\}_{k\in\N}\subset(0,\,\ep_0)$ such that $\ep_k\to 0$ and  
\[
u^{\ep_k}(0,\exp(\gl/\ep_k))\geq \gb_2\quad\text{ for all }\ k\in\N,
\]
while,  \eqref{eq:main5} gives
\[
u^{\ep_k}(0,\exp(\gl_0/\ep_k))\leq \gb_1\quad\text{ for all }k\in\N.
\]

 The continuity of $t \mapsto u^{\ep_k}(0,t)$ implies that,  for each $k\in\N$, 
there exist $\mu_k, \gl_k\in[\gl_0,\, 
\gl]$ such that $\gl_0\leq \mu_k<\gl_k\leq\gl$ and 
\[
u^{\ep_k}(0,\exp(\mu_k/\ep_k))=\gb_1 \ \ \text{ and } \ \ u^{\ep_k}(0,\exp(\gl_k/\ep_k))=\gb_2. 
\]

Proposition  \ref{thm:A} now assures that $\limsup_{k\to\infty}\gl_k\geq M(\gb_2)$, 
but this contradicts that $\gl_k\leq \gl<M(\gb_2)$ for all $k\in\N$.  
\smallskip

 A  similar  argument shows that 
\[
\liminf_{\ep \to 0}u^\ep(0,\exp(\gl/\ep))\geq \bar c(\gl),
\]
and,  thus,  we have \eqref{eq:main2} in the case where $\gl<M(c_0)$. 
\smallskip

 Next we consider the case where $\gl\geq M(c_0)$ and $c_1=c_0$ and recall that, by definition, 
$\bar c(\gl)=c_0$. We first suppose that
\[
\limsup_{\ep\to 0}u^\ep(0,\exp(\gl/\ep))>c_0, 
\]
and use \eqref{FK12} and the upper semicontinuity of $G^+$, to 
select $\gb_2\in\R$ 
so that  $c_0<\gb_2<\limsup_{\ep\to 0}u^\ep(0,\exp(\gl/\ep))$  
 and $G^+(\gb_2)<\gb_2$.
 
Choosing,  for instance, $\gb_1=(c_0+\gb_2)/2$, so that 
$c_0<\gb_1<\gb_2$, and, using Lemma~\ref{lem:main} 
as in the previous case, we may choose sequences 
$\{\ep_k\}$,  
$\{\mu_k\}$, $\{\gl_k\}$ such that  $\lim_{k \to \infty} \ep_k=0$ 
and for some $\gl_0>0$ and all $k\in\N$,  
\[\gl_0\leq \mu_k<\gl_k\leq\gl,\quad 
u^{\ep_k}(0,\exp(\mu_k/\ep_k))=\gb_1\quad\text{ and }\quad
u^{\ep_k}(0,\exp(\gl_k/\ep_k))=\gb_2.
\]
This contradicts  
Proposition ~\ref{thm:B}, and thus, we conclude that 
\[
\limsup_{\ep\to 0}u^\ep(0,\exp(\gl/\ep))\leq c_0.
\]

A similar argument shows 
\[
\liminf_{\ep\to 0}u^\ep(0,\exp(\gl/\ep))\geq c_0,
\]
and, hence, we have \eqref{eq:main2} when $\gl\geq M(c_0)$ and $c_1=c_0$. 
\smallskip

 Now we consider the case where $\gl\geq M(c_0)$ and $c_1>c_0$. 
The definition of $c_1$ implies that $G^-(c)>c\,$ for all $\,c\in[c_0,\,c_1)$, and, 
moreover, by the definition of $\bar c$, we have   
$\,\bar c(\gl)\in[c_0,\,c_1]$, $\gl>M(c)$ for all 
$c\in[c_0,\,\bar c(\gl))$, and, if $\bar c(\gl)<c_1$, then 
$M(\bar c(\gl))=\gl$. 

Suppose that 
\[
\limsup_{\ep\to 0}u^\ep(0,\exp(\gl/\ep))>\bar c(\gl).
\] 

We assume first that
$\bar c(\gl)=c_1$ and observe that we must  have $c_1<g_{\max}$.  
Then \eqref{FK12}   
yields  $\gb_2\in\R$ so that $G^+(\gb_2)<\gb_2$ 
and 
$c_1<\gb_2<\limsup_{\ep\to 0}u^\ep(0,\exp(\gl/\ep))$. 
Fixing $\gb_1\in(c_1,\,\gb_2)$, we argue, as in the previous case, with $c_1$ in place of $c_0$  
and find sequences $\ep_k \to 0+$, 
$\{\mu_k\}$ and $\{\gl_k\}$, and constants $\gl_0>0$ and $\gd>0$ 
such that for all $k\in\N$,  
\[
\gl_0\leq \mu_k<\gl_k\leq\gl,\quad  
u^{\ep_k}(0,\exp(\mu_k/\ep_k)=\gb_1,\quad
u^{\ep_k}(0,\exp(\gl_k/\ep_k))=\gb_2,
\]
which  contradicts Proposition~\ref{thm:B}. 
\smallskip

 Assume next that $\bar c(\gl)<c_1$. 
As noted above, we have $M(\bar c(\gl))=\gl$ and $M(c)<\gl$ for all $c\in[c_0,\,\bar c(\gl))$, 
and, in particular, 
\beq\lab{eq:main6}
M(c)\leq \gl \ \ \text{ for all }\ c\in[c_0,\,\bar c(\gl)].  
\eeq 

Since the function $\bar c$ is continuous at $\gl$, we may choose $\eta>0$ so that 
$\bar c(r)<c_1$ for all $r\in[\gl,\,\gl+\eta]$ and noting that,  for any $r\in(\gl,\,\gl+\eta]$, 
$r> M(c_0)$, we find by the definition of $\bar c(r)$ that 
$M(\bar c(r))=r$, which together with \eqref{eq:main6} implies that $\bar c(r)> \bar c(\gl)$. 

We choose $\gamma\in(0,\,\eta)$ small enough 
so that $\bar c(\gl+\gamma)<\limsup_{\ep\to 0}u^\ep(0,\exp(\gl/\ep))$. 
If we set $\gb_2=\bar c(\gl+\gamma)$ and fix $\gb_1\in(\bar c(\gl),\gb_2)$, then we have 
$\bar c(\gl)<\gb_1<\gb_2<\limsup_{\ep\to 0}u^\ep(0,\exp(\gl/\ep))$. 

 As before, we choose sequences $\{\ep_k\}$, 
$\{\mu_k\}$ and $\{\gl_k\}$ such that $\lim_{k \to \infty} \ep_k=0$ and, for some $\gl_0>0$ and for all $k\in\N$,
\[\gl_0\leq \mu_k<\gl_k\leq\gl,\quad
u^{\ep_k}(0,\exp(\mu_k/\ep_k))=\gb_1 \ \text{and} \ 
u^{\ep_k}(0,\exp(\gl_k/\ep_k))=\gb_2.
\]
Then Proposition \ref{thm:A} imply that $M(\gb_2)\leq 
\limsup_{k\to\infty}\gl_k\leq \gl$. 
On the other hand, we have $\,M(\gb_2)=M(\bar c(\gl+\gamma))=\gl+\gamma>\gl$. 
Hence we obtain a contradiction, 

Thus, in the case when $\gl\geq M(c_0)$ 
and $c_1>c_0$, we have 
\[
\limsup_{\ep\to 0}u^\ep(0,\exp(\gl/\ep))\leq \bar c(\gl),
\] 
while,
by Lemma \ref{lem:main'}, we find 
\[
\liminf_{\ep\to 0}u^\ep(0,\exp(\gl/\ep))\geq \bar c(\gl), 
\]  
and we  conclude that \eqref{eq:main2} holds when $\gl\geq M(c_0)$ and $c_1>c_0$. 
\smallskip

 A similar  argument proves that \eqref{eq:main2} holds when
$\gl\geq M(c_0)$ and $c_1<c_0$,  
and the proof is complete. 
\epr

\appendix

\def\thesection{\Alph{section}}

\section{A subsolution property}
 
 For $T>0$ and 
a (relatively) open subset $\Pi$
of $\pl\gO$, we consider the problem
\beq\lab{eq:A-transport}
\bcases
U_t\leq b(x)\cdot DU \ \  \ \text{ in  }  \ \  \gO\tim(0,\,T],\wcr
\min\{U_t -b(x)\cdot DU,\, U\} \leq 0\ \ \text{ on } \ \  \Pi\tim(0,\, T].
\ecases 
\eeq  

\begin{lemma} \lab{lem:A-transport}
Let $U\in\USC(\lbar Q_T)$ be a subsolution of \eqref{eq:A-transport}, fix $z\in\gO^\Pi$ and set 
\[
u(t)=U(X(T-t,z),t) \quad\text{ for }t\in[0,\,T]. 
\]
Then $u\in\USC([0,\,T])$ and, if $z\in\gO$, it is a subsolution of 
\beq\lab{eq:A-visco1}
u'\leq 0 \quad\text{ in }\ (0,\,T]
\eeq
and, if   $z\in\Pi$, it  is a subsolution of 
\beq\lab{eq:A-visco2}\bcases
u'\leq 0 \  \text{ in } \ (0,\,T), \\[1mm]
\min \{u', u\}\leq 0 \ \text{ on  } \ \{T\}. 
\ecases
\eeq
\end{lemma}

 We note that observations like the lemma above 
concerning 
the restriction of viscosity solutions to lower dimensional manifolds 
go back to Crandall and Lions \cite{CL1983}*{Proposition~I.13}. 

\def\x{\hat x} 
\def\t{\hat t} 

\bpr Let $\phi\in C^1((0,\,T])$ 
and assume that $u-\phi$ has a strict maximum at $\t\in(0,\,T]$.  
\smallskip

 For  $\ga>0$ consider the function 
$\Phi: \lbar Q_T \to \R$ given by 
\[
\Phi(x,t):=U(x,t)-\phi(t)-\ga|x-X(T-t,z)|^2, \ \ 
\]  
let $(x_\ga,t_\ga)\in\lbar Q_T$ be a
maximum point of $\Phi$, set $\x=X(T-\t,z)$, 
and observe that,  as $\ga\to \infty$, 
$(x_\ga,t_\ga)\to (\x,\t)$,  
$\ga|x_\ga-X(T-t_\ga,z)|^2\to 0$ and $U(x_\ga,t_\ga)\to U(\x,\t)$.

 Then, for  $\ga$ sufficiently large, we may assume that 
$(x_\ga,t_\ga)\in\gO\tim (0,\, T]$ 
if either $z\in\gO$ or $\t<T$, 
and  $(x_\ga,t_\ga)\in\gO^\Pi\tim(0,\,T]$ 
if $z\in\Pi$. 
\smallskip

 If $(x_\ga,t_\ga)\in \varOmega\tim(0,\,T]$,  
\eqref{eq:A-transport} yields 
\[
\phi'(t_\ga)-2\ga(X(T-t_\ga,z)-x_\ga)\cdot \dot X(T-t_\ga,z)
\leq 2\ga b(x_\ga)\cdot(x_\ga-X(T-t_\ga,z)),
\]
and then 
\[\begin{aligned}
\phi'(t_\ga)&\,\leq 2\ga (x_\ga-X(T-t_\ga,z))\cdot (b(x_\ga)-b(X(T-t_\ga,z)))
\\&\,\leq 2\|Db\|_{L^\infty(\gO)}\ga|x_\ga-X(T-t_\ga,z))|^2. 
\end{aligned}
\]
\smallskip

Similarly, if  $(x_\ga,t_\ga)\in\Pi\tim(0,\,T]$, then we get 
\[
\phi'(t_\ga)\leq 2\|Db\|_{L^\infty(\gO)}\ga|x_\ga-X(T-t_\ga,z))|^2 
\quad\text{ or }
\quad U(x_\ga,t_\ga)\leq 0.
\]

Sending $\ga\to\infty$ yields 
\[
\phi'(\t)\leq 0 \quad 
\text{ if either }\ z\in\gO \ \text{ or }\ \t<T,
\]
and 
\[
\phi'(\t)\leq 0 \quad\text{ or }\quad 
u(\t)\leq 0 \quad  
\text{ if }\ z\in\Pi\ \text{ and } \ \t=T.  
\]
\epr

\section{The supersolution property up to the boundary}

 For $\ga\in C(\lbar\gO,\,\bbS^n(\gth_0))$ and $H(x,p)=\ga(x)p\cdot p+b(x)\cdot p$ 
 we consider the 
 equation 
\beq\lab{eq:A-super0}
H(x,Du)=0\quad\text{ in }\ \gO.
\eeq

\begin{lemma}\lab{lem:A-super} The maximal subsolution $V\in\Lip(\lbar\gO)$ of \eqref{eq:A-super0} with $V(0)=0$ satisfies, in the viscosity sense, 
\[
H(x,DV)\geq 0\quad\text{ on }\ \lbar\gO.
\]
\end{lemma}
 
 Note that the importance of the lemma above is that the viscosity inequality holds 
up to the boundary. 

\bpr Let $\phi\in C^1(\lbar\gO)$ and assume that 
$V-\phi$ has a strict minimum at $\x \in \lbar\gO$ and $V(\x)=\phi(\x)$.  
\smallskip

 To prove the assertion of the lemma, we argue by contradiction and suppose that   
$H(\x,D\phi(\x))<0$. 
\smallskip

 Indeed, if $\x=0$, then 
\[
H(\x,D\phi(\x))=\ga(0)D\phi(0)\cdot D\phi(0)\geq 0,
\]
and, henceforth, we may assume that $\x\not=0$. 
\smallskip

 We may choose constants $r>0$ and $\ep>0$ so that $0\not\in B_r(\x)$ and 
 \begin{align}\lab{eq:A-super1}
&H(x,D\phi(x))\leq 0 \quad \text{ for all }\ x\in\lbar\gO\cap B_r(\x),\\
&\ep +\phi(x)<V(x)\quad \text{ for all }\ x\in\lbar \gO\setminus B_r(\x).
\lab{eq:A-super2}
\end{align}

It follows from \eqref{eq:A-super1} that, in the viscosity sense, 
\[
H(x,D\phi)\leq 0 \quad\text{ in }\ \gO\cap B_r(\x).
\]

Set 
\[
W(x)=\max\{V(x),\,\ep+\phi(x)\}\quad\text{ for }\ x\in\lbar\gO,
\]
and observe that
$\,\gO=N\cup M$, 
where $\,N:=\gO\cap B_r(\x)$ and $\,M:=\{x\in\gO\mid V(x)>\ep +\phi(x)\}$ 
(note that $N,\,M$ are both open subsets of $\gO$), 
\[
H(x,DW)\leq 0 \quad\text{ in }\ N \ \ \text{  in the viscosity sense},
\] 
$W=V$ in $M$ and $\,\x\in M$. 
Hence, $W$ is a subsolution of \eqref{eq:A-super0}  
such that $W(0)=V(0)=0$ and 
$W(\x)>V(\x)$, which  contradicts the maximality of $V$.  
\epr

\section{A comparison theorem}  We follow the arguments of  
\cite{Is1989}*{Corollary 2.2 \& Remark 2.4} to give
a proof of following lemma.

\begin{lemma} \lab{lem:A-comparison} Let $\ga\in C(\R^n,\bbS^n(\gth_0))$ and 
$H(x,p)=\ga(x)p\cdot p+b(x)\cdot p.$ 
If 
$v\in\Lip(\lbar\gO)$ and $w\in\LSC(\lbar\gO)$ are respectively a subsolution 
and a supersolution of 
the state-constraints problem
\[
H(x,Du)=0\quad\text{ in }\ \gO,
\]
that is, $v$ and $w$ satisfy, respectively,
\[
H(x,Dv)\leq 0 \quad\text{ in }\gO \quad\text{ 
and }\quad H(x,Dw)\geq 0 \quad\text{ on }\ \lbar\gO,\]
and  $v(0)\leq w(0)$, then  \ $u\leq v$ on $\lbar\gO$. 
\end{lemma} 

 Note that the viscosity property of $v$ and $w$ at the origin is indeed not 
required in the lemma above. That is, it is enough to assume that $v$ and $w$ 
are a subsolution of 
\[
H(x,Dv)\leq 0 \quad\text{ in }\ \gO\setminus\{0\},
\]
and a supersolution of 
\[
H(x,Dw)\geq 0 \quad\text{ on }\ \lbar\gO\setminus\{0\}.
\]

\bpr Fix  $\ep>0$ and choose $r\in(0,\,r_0)$ sufficiently small so that
\[
\max_{\pl B_r}v\leq \min_{\pl B_r} w+\ep,
\]
set $\gO(r):=\gO\setminus\lbar B_r$, define $h\in C(\pl \gO(r))$ 
and $v_\ep\in\Lip(\lbar\gO)$ 
by 
\[
v_\ep=v-\ep \ \ \text{and} \ \ h(x)=\bcases
\min_{\pl B_r} w&\text{ if }x\in\pl B_r,\\[1mm] 
\max_{\pl\gO} v&\text{ if }x\in\pl\gO,
\ecases
\]
and observe that $v_\ep$ and $w$ are, respectively, a subsolution and a supersolution 
of the Dirichlet problem in the viscosity sense (see \cite{Is1989}):
\[
\bcases
H(x,Du)=0 \quad\text{ in }
\ \gO(r),&\\ 
u=h \quad\text{ or }\quad H(x,Du)=0 \quad\text{ on }\ \pl\gO(r).&
\ecases
\]  

 It follows from  \cite{IS2015}*{Corollary 4} that there exists $\psi\in
\Lip(\lbar\gO(r))$ which is a subsolution of $H(x,D\psi)\leq -\eta$ 
in $\gO(r)$ for some $\eta>0$ and note that we may assume by adding, if necessary,  
a constant that $\psi\leq v_\ep$ on $\gO(r)$. 
 
  Define $v^\ep\in\Lip(\lbar\gO(r))$ by 
$v^\ep(x)=(1-\ep)v_\ep(x)+\ep\psi(x)$  
and note that 
$v^\ep$ is a subsolution of 
\[
\bcases
H(x,Du)\leq -\ep\eta \quad\text{ in }
\ \gO(r),&\\ 
u\leq h \quad\text{ or }\quad H(x,Du)\leq-\ep\eta \quad\text{ on }\ \pl\gO(r).&
\ecases
\] 

It is clear that the domain $\gO(r)$ satisfies the uniform interior cone condition and, hence, 
we apply \cite{Is1989}*{Corollary 2.2 \& Remark 2.4} to $v^\ep$ and $w_\ep$,
to conclude that $v^\ep\leq w_\ep$ in $\lbar\gO(r)$, from which, after  sending 
$\ep\to 0$, we get $\,v\leq w\,$ on $\,\lbar\gO$. 
\epr

\bye